\documentclass{article}
\usepackage{fancyhdr}
\usepackage{amsmath, amsfonts, amssymb, amsthm}
\usepackage{graphicx}
\usepackage{mathtools}
\usepackage[shortlabels]{enumitem}
\usepackage{color}
\usepackage{eufrak}
\usepackage{verbatim}
\usepackage{sectsty}
\usepackage{amsopn}
\usepackage{setspace}
\usepackage{comment}
\usepackage{tasks}
\usepackage{bbm}
\usepackage{relsize}
\usepackage[english]{babel}
\usepackage{ mathrsfs }
\usepackage{authblk}

\newtheorem{thm}{Theorem}[section]
\newtheorem{definition}[thm]{Definition}
\numberwithin{equation}{section}
\numberwithin{figure}{section}
\newtheorem{corollary}[thm]{Corollary}
\newtheorem{lemma}[thm]{Lemma}

\begin{document}

\title{Maximizers for Strichartz Inequalities on the Torus}

\author[1]{Oreoluwa Adekoya}
\author[2]{John P. Albert}
\affil[1]{Department of Applied Mathematics, University of Washington}
\affil[2]{Department of Mathematics, University of Oklahoma}

\date{}

\maketitle

\begin{abstract} We study the existence of maximizers for a one-parameter family of Strichartz inequalities on the torus. In general maximizing sequences can fail to be precompact in $L^2(\mathbb T)$, and maximizers can fail to exist. We provide a sufficient condition for precompactness of maximizing sequences (after translation in Fourier space), and verify the existence of maximizers for a range of values of the parameter.  Maximizers for the Strichartz inequalities correspond to stable, periodic (in space and time) solutions of a model equation for optical pulses in a dispersion-managed fiber.
\end{abstract}
\section{Introduction}

In this paper we study the existence of maximizers for a one-parameter family of Strichartz inequalities on the torus $\mathbb T = \mathbb R/(2\pi \mathbb Z)$.   These inequalities give bounds on the space-time norms of $T_tu(x)$, where $T_t$ denotes the unitary semigroup defined on $L^2(\mathbb T)$ by the linear Schr\" odinger equation.   That is, for each function $u \in L^2([0,2\pi])$, $T_t f(x)$ is defined to equal $v(x,t)$, where $v$ is the solution of the linear Schr\"odinger equation $v_t - iv_{xx}=0$ on $[0,2\pi]$ with periodic boundary conditions and with initial condition $v(x,0)=u(x)$.   The Strichartz inequalities in question state that for each $B>0$ there exists a constant $C >0$ such that for all $u \in L^2(\mathbb T)$,
\begin{equation}
\label{Strichinequal}
\left(\int_0^B \int_{\mathbb T} |T_tu(x)|^4\ dx\ dt\right)^{1/4} \le C\left(\int_{\mathbb T}|u(x)|^2\ dx\right)^{1/2}.
\end{equation} 
 
 Inequality \eqref{Strichinequal} was proved by  Bourgain in \cite{bourgain} for $B=2\pi$; the result for general $B$ follows immediately from the result for $B=2\pi$ via H\" older's inequality and the fact that $T_t$ is periodic in $t$ with period $2\pi$.
 
 Let $C_B$ denote the best constant in \eqref{Strichinequal}, 
 $$
 C_B = \inf \{C > 0:  \ \text{inequality \eqref{Strichinequal} holds for all $u \in L^2(\mathbb T)$}\}.
 $$
 We say that $u_0 \in L^2(\mathbb T)$ is a maximizer for \eqref{Strichinequal} if 
 $$
 \left(\int_0^B \int_{\mathbb T} |T_tu_0(x)|^4\ dx\ dt\right)^{1/4} = C_B\left(\int_{\mathbb T}|u_0(x)|^2\ dx\right)^{1/2}.
 $$
 It is not hard to see (cf.\ Corollary \ref{nonexistence}  below)  that if $B$ is of the form $B = N \pi$ with $N \in \mathbb N$, then there are no maximizers for \eqref{Strichinequal} in $L^2(\mathbb T)$.   Here, we prove as our main result a sufficient condition for the existence of maximizers for \eqref{Strichinequal}.  As a consequence we obtain that maximizers for \eqref{Strichinequal} do exist at least for $B$ in the range $0 < B < B_4$, where  $B_4 \approx 2.60$.   It remains open whether there is some $B$ in $(0,\pi)$ for which maximizers do not exist.

As explained in \cite{kunze}, an important mathematical feature of the problem of finding maximizers for the analogue of \eqref{Strichinequal} in $L^2(\mathbb R)$ is a loss of compactness:  in general,  it need not be true that maximizing sequences have strongly convergent subsequences in $L^2(\mathbb R)$,  even when maximizers do exist.   In \cite{kunze}, this difficulty is overcome by showing that if  $\{u_j\}$ is a maximizing sequence, then for some subsequence $\{u_{j_k}\}$, $\{e^{i\theta_k}u_{j_k}(x -x_k)\}$ can be made to converge for suitable choices of the sequences $\{\theta_k\}$ and $\{x_k\}$.   In other words, maximizing sequences do always have subsequences which, after being suitably translated in physical space and in Fourier space, converge strongly in $L^2(\mathbb R)$.

In the periodic case, there is a similar difficulty due to loss of compactness of maximizing sequences $\{u_j\}$.   We show below  (see Theorem \ref{mainthm}) that if $C_B > B/\pi$, then every 
maximizing sequence $\{u_j\}$ for \eqref{Strichinequal} must have subsequences $\{u_{j_k}\}$ such that $\{e^{i\theta_k x}u_{j_k}(x)\}$ converges strongly in $L^2(\mathbb T)$, for suitably chosen $\{\theta_k\}$.  Our proof follows the framework of that given for the nonperiodic case in \cite{kunze}: first, concentration-compactness arguments are used to show that maximizing sequences must have subsequences which, after translation, are simultaneously tight in physical space and in Fourier space, after which a decomposition of the translated subsequence into high- and low-frequency parts is used to deduce strong convergence in $L^2(\mathbb R)$.  However,  the application of this technique to the problem on the torus runs into a difficulty which is not encountered for the problem on the line: in the periodic case, for certain values of $B$ (including $B=2\pi$) maximizing sequences can vanish or exhibit splitting in Fourier space, while this cannot happen for maximizing sequences on the line.  That the difficulty is essential, and not just an artifact of the method of proof, is shown by the fact that, as mentioned above, maximizers do not exist in the periodic case for certain values of $B$. This seems to be an instance of the general principle that the effects of dispersion in wave propagation are more subtle and delicate in the periodic case than on the line.

The sufficient condition for maximizers presented in Theorem \ref{mainthm} can be verified by finding an appropriate test function:  it suffices to find $w \in L^2(\mathbb T)$ such that the ratio of $\|T_t w(x)\|_{L^4(\mathbb T\times [0,B])}$ to $\|w\|_{L^2(\mathbb T)}$  is greater than $(B/\pi)^{1/4}$.   We restate this condition on the test function $w$ in a more convenient form in Corollary \ref{ABcriterion}  below,  and by finding suitable test functions which satisfy this condition, we prove the existence result mentioned above. 

Maximizers for the Strichartz inequality also correspond to ground-state solutions of a equation, sometimes known as the dispersion-managed nonlinear Schr\"odinger equation (DMNLS), which models nonlinear, long-wavelength light pulses in a dispersion-managed optical fiber.   In the non-periodic case, where pulses are defined on the entire real line and decay as $|x| \to \infty$ (note that $x$ is actually a time variable in this model), the DMNLS equation was derived in \cite{gabitov} (see also \cite{ab}), and the existence of  ground-state solutions was proved by variational methods in \cite{zharnitsky2} for the case of positive average dispersion, and in \cite{kunze} for the case of zero average dispersion.

The maximizers whose existence is proved in the present paper, by contrast, correspond to solutions of an equation which models  pulses in a dispersion-managed fiber which are periodic in both $x$ and $t$.  This periodic DMNLS equation was derived in  \cite{adekoya}, where well-posedness results for the initial-value problem are proved for the case of positive and zero average dispersion, and results on the existence and stability of periodic ground-state solutions were proved in both cases.   

  In the case of zero average dispersion, the periodic DMNLS equation, for complex-valued functions $u(x,t)$ which are periodic with period $L$ in the $x$ variable, can be written in Hamiltonian form as 
\begin{equation}
u_t = -i \nabla H_L(u).
\label{DMNLS}
\end{equation}
  Here the Hamiltonian functional $H_L:L^2(\mathbb T) \to \mathbb R$ is given by
$$
H_L(u)=-\frac{2\pi}{L}\int_0^L\int_0^1 |T^L_t u(x)|^4\ dt \ dx,
$$  
and $\nabla H_L$ denotes the gradient of $H_L$, given by
$$
\nabla H_L(u)=-\frac{8\pi}{L}\int_0^L\int_0^1 T^L_{-t}\left( |T^L_t u(x)|^2 T^L_t u(x)\right)\ dt \ dx.
$$ 
The operator $T^L_t$ appearing in the integrand is the solution operator for the linear Schr\"odinger equation with periodic boundary conditions on $0\le x \le L$.   That is, $T^L_t(u)(x)=v(x,t)$, where $v(x,t)$ is periodic with period $L$ in $x$ and satisfies the equation $iv_t + v_{xx}=0$, with initial condition $v(x,0)=u(x)$. 
(The gradient here is defined with respect to the real-valued inner product $\langle u,v\rangle$ defined on $L^2(\mathbb T)$ by
$$
\langle u, v\rangle = \Re \int_{\mathbb T} u(x)\overline v(x)\ dx.
$$
That is, we have 
$$
\lim_{\epsilon \to 0} \frac{H_L(u+\epsilon v)-H_L(u)}{\epsilon}= \langle \nabla H_L(u),v\rangle
$$
for all $v \in L^2(\mathbb T)$.)

As an immediate consequence of our results on existence of maximizers for \eqref{Strichinequal}, we obtain results on the existence and stability of sets of ground-state solutions to the periodic DMNLS equation \eqref{DMNLS}, for a range of values of the period $L$.  Ground-state solutions can be characterized as solutions of the form $u(x,t)=e^{i\omega t}\phi(x)$, where $\omega \in \mathbb R$ and the profile function $\phi(x)$ minimizes $H_L(u)$ among all functions in $L^2(0,L)$ with fixed $L^2$ norm $\lambda$.  For each fixed value of $\lambda > 0$, the stability of the set $S_{L,\lambda}$ of corresponding ground-state profile functions follows from a standard argument, once we have shown that every minimizing sequence for the associated variational problem converges strongly to $S_{L,\lambda}$ in $L^2$ norm. 

The organization of this paper is as follows.   In Section \ref{sec:notation}, we establish notation and state our main results.   Section \ref{sec:prelim} contains some preliminary lemmas.  The proof of Theorem \ref{mainthm}, on the sufficiency of the condition $C_B>B/\pi$ for the existence of maximizers, is given in Section \ref{sec:proofmainthm}.   This sufficient condition is verified for a range of values of $B$ in Section \ref{sec:existence}.  The final Section \ref{sec:stability} discusses the implications for existence and stability of non-empty sets of ground-state solutions of the periodic DMNLS equation \eqref{DMNLS}.

 The results of this paper are taken from the first author's doctoral dissertation at the University of Oklahoma \cite{adekoya}.

\section{Notation and statement of main results.} 
\label{sec:notation}

If $E$ is a measurable subset of $\mathbb R$ and $1 \le p < \infty$, we define $L^p(E)$ to be the space of Lebesgue measurable complex-valued functions $u$ on $E$ such that 
$\|u\|_{L^p(E)}=\left(\int_E |u|^p\ dx\right)^{1/p}$ is finite.  We denote by $L^2(\mathbb T)$ the space of Lebesgue measurable, square-integrable, $2\pi$-periodic functions on  $\mathbb R$.   We can identify $L^2(\mathbb T)$ with $L^2([0,2\pi])$.  We will often denote the norm of $u$ in $L^2(\mathbb T)$ simply by $\|u\|_{L^2}$.  For $B>0$, we define $L^p_{t,x}([0,B]\times \mathbb T)$ to be the space of all functions $f(t,x)$ defined for $(t,x) \in [0,B]\times \mathbb T$ such that the norm 
$\|f\|_{L^p_{t,x}([0,B]\times \mathbb T)}=\left(\int_0^B\int_{\mathbb T}|f(t,x)|^p\ dx\ dt\right)^{1/p}$ is finite.

For $1 \le p < \infty$, we define $\ell^p(\mathbb Z)$ to be the space of sequences of complex numbers $\{a(n)\}_{n \in \mathbb Z}$ such that
$\|a\|_{\ell^p} = \left(\sum_{n \in \mathbb Z}|a(n)|^p\right)^{1/p}$ is finite.  We define $\ell^\infty(\mathbb Z)$ to be the space of all sequences $\{a(n)\}_{n \in \mathbb Z}$ such that
$\|a\|_{\ell^\infty}=\sup_{n \in \mathbb Z}|a(n)|$ is finite.

For $u \in L^2(\mathbb T)$, we define the Fourier transform of $u$ to be the sequence $\mathcal{F}u$ in $l^2(\mathbb Z)$ given by
$$\mathcal{F} u[n]= \frac{1}{2\pi}\int_{\mathbb{T}} e^{-inx} u(x)\ dx$$
for $n \in \mathbb Z$. We also denote $\mathcal{F} u[n]$ by $\hat f(n)$. The inversion formula for the Fourier transform is given by 
$$ u(x)= \sum_{n \in \mathbb{Z}} \hat u(n) \ e^{inx}.$$
The correspondence $u \to \hat u$ defines a one-to-one map from $L^2(\mathbb T)$ onto $\ell^2(\mathbb Z)$; and with this definition of the Fourier transform, Parseval's theorem asserts that for any $u,v \in L^2(\mathbb T)$, one has
$$
\int_{\mathbb T}u(x)\bar v(x)\ dx = 2\pi\sum_{n \in \mathbb Z}\hat u(n)\bar {\hat v}(n),
$$
and in particular
$$
\|u \|_{L^2}=\sqrt {2\pi}\|\hat u\|_{\ell^2}.
$$
Also, the Fourier transform of the product $uv$ is given by a convolution:
$$
\mathcal{F}(uv)[n]=(\hat u * \hat v) [n] = \sum_{k \in \mathbb Z} \hat u(k)\hat v(n-k).
$$

In Section \ref{sec:stability}, we will have occasion to mention the action of the Fourier transform on functions of period $L$.   Define $L^2_{\rm per}(0,L)$ to be the set of all measurable functions on $\mathbb R$ which are periodic of period $L$ and which are square integrable on $0 \le x \le L$.   For $u \in L^2_{\rm per}(0,L)$, we define the Fourier transform $\mathcal F_L u \in \ell^2(\mathbb Z)$ by 
$$
\mathcal{F}_Lu[n]= \frac{1}{L}\int_0^L e^{-i(2\pi n/L)x} u(x)\ dx,
$$
and we have the have the Fourier inversion formula
$$
u(x)=\sum_{n \in \mathbb Z} \mathcal F_L u[n]\ e^{i(2\pi n/L)x}.
$$

We define the Sobolev space $H^1=H^1(\mathbb T)$ to be the space of all functions $u \in L^2(\mathbb T)$ such that the $H^1$ norm
$$
\|u\|_{H^1}=\left(\sum_{n \in \mathbb Z} |n|^2|\hat u(n)|^2\right)^{1/2}
$$
is finite.

We denote by $\mathcal D$ the set of all functions $u \in L^2(\mathbb T)$ such that $\hat u(n)=0$ for all but finitely many $n \in \mathbb Z$.   In particular, functions in $\mathcal D$ are infinitely smooth.

For functions $g(t,x) \in L^2_{t,x}([0,2\pi]\times \mathbb T)$,  the space-time Fourier transform of $g$ is the sequence $\mathcal F_{t,x}g \in \mathbb Z \times \mathbb Z$ defined by
$$
\mathcal F_{t,x}g[m,n]=\frac{1}{(2\pi)^2}\int_0^{2\pi}\int_{\mathbb T}e^{-imt}e^{-int}g(t,x)\ dt\ dx.
$$
The correspondence $g \to \mathcal F_{t,x}g$ defines a one-to-one map from $L^2([0,2\pi]\times \mathbb T)$ onto the space $\ell^2(\mathbb Z \times \mathbb Z)$ of square-integrable sequences $b[m,n]$, and Parseval's theorem asserts that for $g_1, g_2 \in L^2_{t,x}([0,2\pi]\times \mathbb T)$, one has
$$
\int_{\mathbb T}\int_0^{2\pi} g_1 \overline {g_2}\ dt\ dx = (2 \pi)^2 \sum_{m \in \mathbb Z} \sum_{n \in \mathbb Z} \mathcal F_{t,x} g_1 \overline{\mathcal F_{t,x}g_2}.
$$

For $t \in \mathbb R$, define $T_t:L^2(\mathbb T) \to L^2(\mathbb T)$ as a Fourier multiplier operator by setting, for $u \in L^2(\mathbb T)$ and $n \in \mathbb Z$,
\begin{equation}
\label{defT}
\mathcal{F}(T_t u)[n]=e^{-in^2 t}\mathcal Fu[n].
\end{equation}
For a given $u \in L^2(\mathbb T)$, $T_tu(x)$ is thus defined as a measurable function of $x$ and $t$.  Since $\int_{\mathbb T}|T_t u(x)|^2\ dx = \|\mathcal{F}(T_t u)\|_{\ell^2}^2=\|\mathcal Fu[n]\|_{\ell^2}^2=\|u\|_{L^2(\mathbb T)}^2$ for each $t \in \mathbb R$, we have $ \int_0^B\int_{\mathbb T}|T_tu(x)|^2\ dt\ dx < \infty$ for every $B>0$.  Therefore $T_tu \in L^2_{t,x}([0,B]\times \mathbb T)$.  In particular, taking $B=2\pi$, we have that $\mathcal F_{t,x}(T_tu)$ is well-defined in $\ell^2(\mathbb Z \times \mathbb Z)$ and is given by
$$
\mathcal F_{t,x}(T_tu)[m,n]=\begin{cases} e^{-in^2t}\hat u(n) \quad &\text{if $m = n^2$}\\
                                                               0 \quad &\text{if $m \ne n^2$}.\end{cases}
$$

Fix $B>0$, and for $u \in L^2(\mathbb T)$,  define
\begin{equation}
W_B(u)=\int_0^B\int_{\mathbb{T}}  |T_t u(x)|^4 \ dx\ dt . 
\end{equation}
We consider the variational problem of maximizing $W_B(u)$ over $L^2(\mathbb T)$, subject to the constraint $\|u\|_{L^2}^2= \lambda$, where  $\lambda >0$ is fixed. 
Define
\begin {equation}\label{jlambda}
J_{B,\lambda}= \sup\ \{W_B(u): u\in L^2(\mathbb T) \text{ and } \|u\|_{L^2}^2= \lambda\}.
\end{equation}

We say that a sequence $\{u_j\}_{j \in \mathbb N}$ in $L^2(\mathbb T)$ is a maximizing sequence for $J_{B,\lambda}$ if $\|u_j\|_{L^2}^2=\lambda$ for all $j \in \mathbb N$ and $\lim_{j \to \infty}W_B(u_j)=J_{B,\lambda}$;
and we say that $u_0 \in L^2(\mathbb T)$ is a maximizer for $J_{B,\lambda}$ if $\|u_0\|_{L^2}^2=\lambda$ and $W_B(u_0)=J_{B,\lambda}$.

Observe that since $W_B(\lambda u)=\lambda^4 u$ for all $\lambda \in \mathbb R$, it follows that 
\begin{equation}
J_{B,\lambda}=\lambda^2 J_{B,1}
\end{equation} 
for all $\lambda >0$.   In other words, if we define $C_B = J_{B,1}$, then for all $u \in L^2(\mathbb T)$ we have
\begin{equation}
W_B(u) \le C_B\|u\|_{L^2}^4,
\label{Strich}
\end{equation}
which is equivalent to the Strichartz inequality \eqref{Strichinequal} with best constant $C_B$.  That $C_B$ is indeed finite is shown below in Lemma \ref{Tsu}.
  It is clear that the existence of a maximizing function for $C_B=J_{B,1}$ is equivalent to the existence of a maximizing function for $J_{B,\lambda}$  for every $\lambda > 0$.
  
  The following theorem establishes a sharp condition for the precompactness of maximimizng sequences for $J_{B,1}$.

\begin{thm}  
\item{(i)} If $B>0$ and 
\begin{equation}
J_{B,1} > \frac{B}{\pi},
\label{maincond}
\end{equation}
then every maximizing sequence for $J_{B,1}$ has a subsequence which, after translations in Fourier space, converges in $L^2(\mathbb T)$ to a maximizer for $J_{B,1}$. That is, if $\{u_j\}$ is a sequence such that $\|u_j\|_{L^2}=1$ for all $j \in \mathbb N$ and $\lim_{j \to \infty}W_B(u_j)= J_{B,1}$, 
then there exists a subsequence $\{u_{j_k}\}$ and a sequence of real numbers $\{\theta_k\}$ such that $\{e^{i\theta_k x}u_{j_k}(x)\}$ converges strongly in $L^2(\mathbb T)$. 

In particular, there does exist a maximizer for $J_{B,1}$: that is, there exists  $u_0 \in  L^2(\mathbb T)$ such that $\|u_0\|_{L^2}=1$ and $W_B(u_0)=J_{B,1}$.   
\bigskip
\item{(ii)} For all $B>0$, $$J_{B,1} \ge \frac{B}{\pi}.$$
\bigskip
\item{(iii)} If $B>0$ and 
\begin{equation}
J_{B,1} = \frac{B}{\pi},
\end{equation}
then there exist maximizing sequences for $J_{B,1}$ which do not have any subsequences that can be made to converge by translating the terms in Fourier space.
\label{mainthm}
\end{thm}

As corollaries of this result, we obtain existence and non-existence results for maximizers of $J_{B,1}$ for certain values of $B$.  In Section 5 below we show that $J_{B,1}>B/\pi$ is true at least for all $B$ in some range $0 < B < B_4$, where $B_4 \approx 2.6$; and therefore $J_{B,1}$ does have maximizers for $B$ in this range (see Corollary \ref{existencerange}). On the other hand, we see that $J_{B,1} = B/\pi$ for all $B$ of the form $B= N\pi$, where $N \in \mathbb N$, and hence for these values of $B$, no maximizer for $J_{B,1}$ exists (see  Corollary \ref{nonexistence}).  

\bigskip
{\it Remark.} Although the questions of whether maximizers exist for $J_{B,1}$, and whether $J_{B,1}$ is equal to $B/\pi$, are somewhat subtle;  it is easy to answer the corresponding questions for {\it minimizers} of $W_B(u)$ subject to the constraint that $\|u\|_{L^2}=1$.  In fact, for every $B>0$ the minimum is equal to $B/2\pi$, and is attained at the constant function $v(x) \equiv 1/(2\pi)$ on $\mathbb T$.   To see this, note that by H\"older's inequality, if $\|u\|_{L^2}=1$ then
$$
B=\int_0^B\int_{\mathbb T} |u|^2\ dx\ dt = \int_0^B\int_{\mathbb T}|T_tu|^2\ dx\ dt \le \sqrt{W_B(u)} \sqrt{2\pi B},
$$
which implies that $W_B(u) \ge B/(2\pi) = W_B(v)$.

\bigskip
  
 By a well-known argument, the assertions of Theorem \ref{mainthm} yield results on the existence and stability of sets of ground-state solutions of the periodic DMNLS equation \eqref{DMNLS}, for a range of values of the period $L$.  We review these arguments below in Section 6, where we show (see Theorem \ref{stabilitythm}) that for all $L \in (0,2\pi/\sqrt{B_4})$, equation \eqref{DMNLS} has a one-parameter family $\{S_{L,\lambda}: \lambda > 0\}$ of non-empty sets $S_{L,\lambda}$ of ground-state profiles, and each set $S_{L,\lambda}$ is stable with respect to the flow defined by \eqref{DMNLS}.
 
   On the other hand, the nonexistence of maximizers of $J_{B,1}$ when $B$ is an integer multiple of $\pi$ translates into a nonexistence result for ground-state solutions of \eqref{DMNLS}:   when $L$ is of the form $L = 2\sqrt{\pi/N}$ for some $N \in \mathbb N$, then \eqref{DMNLS} can have no ground-state solutions (see Theorem \ref{nogroundstate}).

\section{Preliminary results}
\label{sec:prelim}

An important property of $W_B$ is that it is invariant with respect to translations in Fourier space as well as translations in physical space.  

\begin{lemma}  
\label{invariant}   Let $B>0$.
\item{(i)}  Suppose $u \in L^2(\mathbb T)$ and $x_0 \in \mathbb T$.  If we define $v \in L^2(\mathbb T)$ by $v(x)=u(x-x_0)$ for $x \in \mathbb T$, then for all $(t,x) \in \mathbb R \times \mathbb T$, we have $T_tv(x)=T_t(x-x_0)$.
In particular,
$$
W_B(v)=W_B(u).
$$
\item{(ii)} Suppose $u \in L^2(\mathbb T)$ and $n_0 \in \mathbb Z$.  If we define $w \in L^2(\mathbb T)$ by setting $\hat w(n)=\hat u(n-n_0)$ for all $n \in \mathbb Z$, then for all $(t,x) \in \mathbb R \times \mathbb T$, we have
$$
T_tw(x)=e^{in_0x}e^{-in_0^2 t} T_tu(x-2n_0t).
$$
In particular
$$
W_B(w)=W_B(u).
$$
\item{(iii)} If $\{u_j\}_{j\in \mathbb N}$ is a maximizing sequence for $J_{B,1}$ in $L^2(\mathbb T)$,  $\{m_j\}_{j \in \mathbb N}$ is a sequence of integers, and $\{x_j\}_{j \in \mathbb N}$ is a sequence in $\mathbb T$, then $\{e^{im_jx}u_j(x-x_j)\}_{j \in \mathbb N}$ is also a maximizing sequence for $J_{B,1}$.    Also, if $u$ is a maximizer for $J_{B,1}$, then $e^{imx}u(x-x_0)$ is also a maximizer, for every $m \in \mathbb Z$ and every $x \in \mathbb T$.
\end{lemma}

\begin{proof}The statements in (i) and (ii) follow easily from the definition of $T_t$ as a Fourier multiplier operator.  We note that the invariance of $W_B$ under translations in Fourier space also follows immediately from the formula given below for $W_B$ in \eqref{WBeqsum}.

Part (iii) of the Lemma follows immediately from parts (i) and (ii), since the norm in $L^2(\mathbb T)$ is also invariant under translations in both physical space and Fourier space.
\end{proof}

We now state a version of Lions' concentration compactness lemma. 

\begin{lemma}\label{conccomp}  Fix $M>0$, and suppose that for each $j \in \mathbb N$, $\{a_j(n)\}_{n \in \mathbb Z}$ is an element of $\ell^2(\mathbb Z)$ such that $\| a_j\|^2_{\ell^2}=M$.  Then the sequence $\{a_j\}_{j \in \mathbb N}$ in $\ell^2(\mathbb Z)$ has a subsequence, still denoted by $\{a_j\}$, for which exactly one of the following three alternatives holds:\\
\begin{enumerate} 

\item  (Vanishing) For every $r\in \mathbb N$,
$$\lim_{j\to \infty}\sup_{m\in \mathbb{Z}} \sum_ {n=m-r}^{m+r} |a_j(n)|^2  =0.$$ 

\item (Splitting) 
There is an $\alpha \in (0,M)$ with the following property:   for every $\delta >0$,  there exist numbers $r_1,r_2 \in \mathbb N$ with $r_2- r_1 \ge \delta^{-1}$,  sequences $\{b_j\}_{j \in \mathbb N}$ and $\{c_j\}_{j \in \mathbb N}$ in $\ell^2(\mathbb Z)$, and an integer sequence $\{m_j\}_{j \in \mathbb N}$   such that for all $j \in \mathbb N$, 
$$b_j(n)=0 \ \text{for all $n \in \mathbb Z$  such that $|n-m_j| >r_1$},$$
$$c_j(n)=0 \ \text{for all $n \in \mathbb Z$ such that $|n-m_j| <r_2$},$$  
$$\|a_j-(b_j+c_j)\|^2_{\ell^2}\le \delta,$$
$$\left| \|b_j\|^2_{\ell^2}-\alpha\right|\le \delta, $$
and
$$\left| \|c_j\|^2_{\ell^2}-(M-\alpha)\right|\le \delta.$$ 

\item (Tightness) There exist integers $\{m_j\}_{j \in \mathbb N}$ such that for every $\epsilon>0,$ there exists $r \in \mathbb N$ so that
$$\sum_ {n=m_j -r}^{m_j +r} |a_j (n)|^2 \; dx \ge M-\epsilon$$ 
for all $j \in \mathbb N$.

\end{enumerate}
\end{lemma}  

We omit the proof of Lemma \eqref{conccomp}, which is standard:  for example, except for obvious modifications it is the same as the proof given for Lemma 3.1 of \cite{kunze}.    However, for future reference we emphasize here that the three alternatives given in Lemma 3.1 are mutually exclusive.  In particular, if there exist integers $\{m_j\}$ such that the translated  sequence  $\{\tilde a_j\} = \{a_j(\cdot-m_j)\}$  converges strongly in $\ell^2(\mathbb Z)$, then all subsequences of $\{a_j\}$ are tight, and no subsequence of $\{a_j\}$ vanishes.   For indeed, if $\{\tilde a_j\}$ converges in $\ell^2$ norm to a limit $a\in \ell^2(\mathbb N)$, then we must have $\|a\|_{\ell^2}^2=M > 0$, and therefore for every $\epsilon > 0$ there exists $r \in \mathbb N$ such that $\sum_{n=-r}^r|a(n)|^2 > M - \epsilon$.  From the strong convergence of $\{\tilde a_j\}$ to $a$ in $\ell^2$, it then follows that
$$
\sum_{n=m_j-r}^{m_j+r}| a_j(n)|^2 > M-\epsilon
$$
for all sufficiently large $j$.  This implies that all subsequences of $\{a_j\}$ are tight, and that no subsequence can vanish.

The following lemma  gives a Fourier decomposition of $W_B(u)$ which will be important in analyzing the behavior of maximizing sequences. 
All sums which appear are intended to be performed over all integral values of the index of summation, unless otherwise specified.

\begin{lemma}\label{Tsu} Suppose $B>0$.  
\item{(i)} There exists $C>0$ such that for all $u \in L^2(\mathbb T)$,
\begin{equation}
\label{GBest}
G_B(u)\le C \|u\|_{L^2(\mathbb T)}^4,
\end{equation}
where
\begin{equation}
\label{defGB}
G_B(u)= 
\sum_l \sum_n \sum_p \frac{\left|\hat u (n)\hat u(n-l) \hat u(n-p)\hat u(n-p-l)\right|}{1+|lp|B} .
\end{equation}

\item{(ii)}   For all $u\in  L^2(\mathbb T)$, we have  
\begin{equation}
\label{WBeqsum}
W_B(u)=2\pi \sum_l \sum_n \sum_p \hat u (n)\bar{\hat u}(n-l) \bar{\hat u}(n-p)\hat{u}(n-p-l)\int_0^B e^{-2ilpt}\ dt ,
\end{equation} 
where the sum on the right-hand side converges absolutely.  

Moreover, there exists $C>0$ such that for all $u \in L^2(\mathbb T)$,
\begin{equation}
\label{WBest}
W_B(u)\le C \|u\|_{L^2(\mathbb T)}^4.
\end{equation}

\item{(iii)} For all $u \in L^2(\mathbb T)$,  we have
\begin{equation}
\label{WandD}
W_B(u)= 4\pi B\|\hat u\|^4_{\ell^2}- 2\pi B\|\hat u\|^4_{\ell^4}+ D_B(u),
\end{equation}
where
\begin{equation}
\label{defDB}
D_B(u)=2\pi \sum_{l \ne 0} \sum_n \sum_{p \ne 0} \hat u (n)\bar{\hat u}(n-l) \bar{\hat u}(n-p)\hat{u}(n-p-l) \int_0^B e^{-2ilpt}\ dt.
\end{equation}
(The sum on the right-hand side converges absolutely.)

 \end{lemma}
 
\begin{proof}
Suppose $u \in L^2(\mathbb T)$.  We can decompose the triple sum which defines $G_B(u)$ into two parts $(I)$ and $(II)$, where $(I)$ represents the sum taken over all $(l,n,p) \in \mathbb Z^3$ for which $|p| < |l|$, and
$(II)$ represents the sum taken over all $(l,n,p)$ for which $|p| \ge |l|$.

Define $K:\mathbb Z \to \mathbb R$ by $K(n)=1/(1+|n|^2B)$, and note that $K \in \ell^1(\mathbb Z)$. If $|p| < |l|$, then we have $1/(1+|lp|B) \le K(p)$.   Therefore we can use Holder's inequality and Young's convolution inequality
to make the estimate
\begin{equation}
\begin{aligned}
(I) &\le \sum_n \sum_p K(p)\left|\hat u (n)\hat u(n-p)\right| \sum_l\left|\hat u(n-l)\hat u(n-p-l)\right|\\
& \le \|\hat u\|_{\ell^2}^2\sum_n\left|\hat u(n)\right|\sum_pK(p)\left|\hat u(n-p)\right|\\
&= \|\hat u\|_{\ell^2}^2\sum_n\left|\hat u(n)\right| (K \ast |\hat u|)(n)\\
&\le \|\hat u\|_{\ell^2}^2\|\hat u\|_{\ell^2}\|K \ast |\hat u|\|_{\ell^2}\\
&\le \|\hat u\|_{\ell^2}^4\|K\|_{\ell^1} \le C\|u\|_{L^2(\mathbb T)}^4.
\end{aligned}
\end{equation} 
On the other hand, when $|p| \ge |l|$, we have $1/(1+|lp|B) \le K(l)$, so we can write
$$
(II) \le  \sum_n \sum_l K(l)\left|\hat u (n)\hat u(n-l)\right| \sum_p\left|\hat u(n-p)\hat u(n-p-l)\right|,
$$
and then use the same argument as for $(I)$, only with $l$ and $p$ interchanged, to show that $(II)\le C\|u\|_{L^2(\mathbb T)}^4$. 
This then proves part (i) of the Lemma.

To prove part (ii), suppose first that $u \in \mathcal D$, the space of all functions $u \in L^2(\mathbb T)$ such that $\hat u$ is compactly supported in $\mathbb Z$,
so that in particular $T_tu$ is bounded on $\mathbb T$ for all $t \in \mathbb R$, and all the computations which follow are readily justified. 
Writing $T_tu=\sum_{n \in \mathbb{Z}} \hat u(n) e^{i (nx -n^2t)}$, we obtain
$$
\begin{aligned}
W_B(u)&=\|T_tu\|^4_{{L_{t,x}^4}([0,B]\times \mathbb{T})}=\| T_tu \cdot\overline {T_tu }\|^2_{{L_{t,x}^2}([0,B]\times \mathbb{T})}\\
&=\left\| \sum_n\sum_m \hat u (n)\bar{\hat u}(m) e^{i ((n-m)x -(n^2-m^2)t)}\right\|^2_{ L^2_{t,x}([0,B]\times \mathbb{T})}\\
&=\left \|  \sum_n \sum_l \hat u (n)\bar{\hat u }(n-l) e^{ilx}e^{ -il(2n-l)t}\right\|^2 _{ L^2_{t,x}([0,B]\times \mathbb{T})},
\end{aligned}
$$
where in the last step we used $l=n-m$ as an index of summation.
Now letting 
$$
b(l,t)= \sum_n \hat u (n)\bar{\hat u}(n-l) e^{ -il(2n-l)t},
$$ 
we can write
$$W_B(u)=\int_0^B \int _\mathbb{T} \left| \sum_l  b(l,t) e^{ilx} \right|^2 \ dx\ dt.$$
Using Parseval's theorem, we obtain that 
\begin{equation}
\label{sum4}
\begin{aligned}
W_B(u)&=2\pi \int_0^B\sum_l |b(l,t)|^2 \ dt \\
&= 2\pi \sum_l \sum_n\sum_r \hat u (n)\bar{\hat u}(n-l) \hat{u}(r-l)\bar{\hat u}(r)  \int_0^B e^{-il(2n-2r)t} \ dt.\end{aligned}
\end{equation}
Changing the index of summation in the innermost sum to $p=n-r$ yields the sum on the right-hand side of \eqref{WBeqsum}.
Thus we have proved that \eqref{WBeqsum} holds, at least in the case when $u \in \mathcal D$. 

In light of the fact that 
\begin{equation}
 \left|\int_0^B e^{-2i\theta t}\ dt\right|  \le \frac{2B}{1+|\theta|B}
 \label{integralest}
\end{equation}
for all $\theta \in \mathbb R$ and all $B>0$, it follows from what we have proved that there exists  $C>0$ such that
$$
W_B(u) \le CG_B(u)  
$$
and therefore
\begin{equation}
\label{boundonTt}
W_B(u) =\|T_t u\|_{L^4_{t,x}([0,B]\times \mathbb T)}^4 \le C \|u\|_{L^2(\mathbb T)}^4
\end{equation}
for all $u \in \mathcal D$. 

Given any $u \in L^2(\mathbb T)$, define a sequence $\{u_N\} \in \mathcal D$ by setting $\widehat{u_N}(n)=\hat u(n)$ for $|n| \le N$ and $\widehat{u_N}(n)=0$ for $|n| > N$.    Then  
\begin{equation}
\label{WBdecompN}
W_B(u_N)= 2\pi \sum_l \sum_n \sum_p \widehat{ u_N} (n)\overline{\widehat{ u_N}}(n-l) \overline{\widehat{ u_N}}(n-p)\widehat{u_N}(n-p-l)\int_0^B e^{-2ilpt}\ dt
\end{equation} holds for each $N \in \mathbb N$. 

By Parseval's theorem, $T_tu_N$ converges to $T_t u$ in $L^2_{t,x}([0,2\pi]\times \mathbb T)$ as $N \to \infty$.  
It follows from \eqref{boundonTt} that $T_tu_N$ also converges to $T_t u$ in $L^2_{t,x}([0,B]\times \mathbb T)$ for every $B \in [0,2\pi]$, 
and then by periodicity for every $B \in \mathbb R$.  Also, by what we have proved, $\{T_tu_N\}_{N \in \mathbb N}$ is a
 Cauchy sequence in $L^4_{t,x}([0,B]\times \mathbb T)$, and so converges in the norm of $L^4_{t,x}([0,B]\times \mathbb T)$ to some limit, which must therefore
 equal $T_tu$.  It follows that
 $$
 W_B(u)=\lim_{N \to \infty} W_B(u_N).
 $$
 On the other hand, it follows from part (i) of the Lemma, \eqref{integralest}, \eqref{WBdecompN}, and the Dominated Convergence Theorem that
 $$
 \lim_{N \to \infty}W_B(u_N)=2\pi \sum_l \sum_n \sum_p \hat u (n)\bar{\hat u}(n-l) \bar{\hat u}(n-p)\hat{u}(n-p-l)\int_0^B e^{-2ilpt}\ dt ,
 $$
 with the sum on the right-hand side converging absolutely.  Therefore part (ii) of the Lemma has been proved. 
 
To prove part (iii), we proceed by splitting the sum in \eqref{WBeqsum} into four parts, according to whether $p$ and $l$ are zero or nonzero. 

First, we sum over all values of $l$, $n$, and $p$ such that $l\ne0$ and $p=0$. This gives
$$
\begin{aligned}
 2\pi B \sum_{l\neq 0} \sum_n\left|\hat u (n)\hat u(n-l)\right|^2&= 2\pi B\sum_n \sum_{m\neq n }\left|\hat u (n)\hat u(m)\right|^2  \\
&= 2\pi B\left(\sum_n \sum_m\left| \hat u (n)\hat u(m)\right|^2 - \sum_n\left|\hat u (n)\right|^4\right) \\
&= 2\pi B\left(\|\hat u\|^4_{\ell^2}-\|\hat u\|^4_{\ell^4} \right). 
\end{aligned}
$$

Second, we sum over all values of  $l$, $n$, and $p$ such that $l=0$ and $p \ne 0$, obtaining the same result as above:  that is,
$$
2\pi B\left(\|\hat u\|^4_{\ell^2}-\|\hat u\|^4_{\ell^4} \right).
$$

Third, we sum over all values of  $l$, $n$, and $r$ such that $l=0$ and $p=0$, resulting in
$$
 2\pi B \sum_n \left|\hat u (n)\right|^4 = 2\pi B\|\hat u\|^4_{\ell^4}. 
 $$

Finally, if we sum over all values of  $l$, $n$, and $p$ such that $l \ne 0$ and $p \ne 0$,   we obtain the sum in \eqref{defDB} which defines $D_B(u)$.  (Note that the absolute convergence of this sum is guaranteed by part (ii) of the Lemma.) Taking the sum of all four parts, we obtain the result
\eqref{WandD}, completing the proof of the Lemma. 
\end{proof}

For what follows, we note that if $\{u_j\}$ is a sequence in $L^2(\mathbb T)$ such that $\|u_j\|_{L^2}^2 = 1$ for all $j \in \mathbb N$, then by Parseval's theorem, we have that
$\{\hat u_j\}$ is a sequence in $\ell^2(\mathbb Z)$ with $\|\hat u_j\|_{\ell^2}^2 = \frac{1}{2\pi}$, and therefore we can apply Lemma \ref{conccomp} to $\{\hat u_j\}$ with $M=\frac{1}{2\pi}$.

\begin {lemma}\label{l4convg} 
Let  $\{u_j\}_{j\in \mathbb N}\subset L^2(\mathbb{T})$ be a sequence such that $\|u_j\|^2_{L^2} = 1$ for all $j \in \mathbb N$.  Suppose that the sequence $\{\widehat{u_j}\}$ in $\ell^2(\mathbb Z)$  vanishes in the sense of Lemma \ref{conccomp}. Then $\|\widehat u_j\|^4_{\ell^4} \to 0$ as  $j\to \infty.$ 
\end{lemma}
\begin{proof}
If $\{\widehat{u_j}\}$ vanishes, then for each $r \in \mathbb N$ and for each $\epsilon >0$, there exists  $N= N(r,\epsilon)\in \mathbb{N}$ such that
$$\sup_{m\in \mathbb{Z}} \sum_{n=m-r}^{m+r}|\widehat{u_j}(n)|^2 < \frac\epsilon2$$ for $j\ge N.$  In particular,  
 for $j\ge N$ we have that $\displaystyle{ |\widehat{u_j}(n)|^2< \frac\epsilon2}$ for all $n \in \mathbb Z$. This implies that $\displaystyle {\|\widehat{u_j}\|_{\ell^\infty}\to 0}$ as  $j\to \infty.$  Since $\displaystyle {\|\widehat{u_j}\|^4_{\ell^4}\le \|\widehat{u_j}\|_{\ell^2}^2\|\widehat{u_j}\|^2_{\ell^\infty}}$,  it follows that $\|\widehat u_j\|^4_{\ell^4} \to 0$ as  $j\to \infty.$ \\

\end{proof}

\begin{definition} For $u_1, u_2, u_3, u_4 \in L^2(\mathbb T)$, define
$$
F(u_1,u_2,u_3,u_4)=\int_0^B \int_{\mathbb T} T_tu_1 \ \overline{ T_tu_2}\ T_tu_3 \ \overline{ T_tu_4}\ dx\ dt.
$$
\end{definition} 

\begin{lemma}\label{multi} There exists $C> 0$ such that
for all functions $u_1$, $u_2$,  $u_3$, and $ u_4$ in $L^2(\mathbb T)$,
$$|F(u_1,u_2,u_3,u_4)|\le C \|u_1\|_{L^2_x} \|u_2\|_{L^2_x} \|u_3\|_{L^2_x} \|u_4 \|_{L^2_x}.$$
\end{lemma}

 \begin{proof} By Holder's inequality and Lemma \ref{Tsu}, we have
 $$
\begin{aligned}
\int_0^B\int_\mathbb{T}\left| T_tu_1 \ \overline{ T_tu_2}\ T_tu_3 \ \overline{ T_tu_4}\right|\ dx\ dt &\le  \|T_tu_1 \|_{L^4_{s,x}}\ \|T_tu_2\|_{L^4_{t,x}}\ \|T_tu_3\|_{L^4_{t,x}} \ \| T_tu_4\|_{L^4_{t,x}}\\
&\le C \|u_1\|_{L^2_x} \|u_2\|_{L^2_x} \|u_3\|_{L^2_x} \|u_4 \|_{L^2_x}.
\end{aligned}
$$
\end{proof}

\begin{lemma}\label{multi2}  There exists $C>0$ such that for all $u,v,w,h \in L^2(\mathbb{T})$ with $ u= v+w+h$ and $\|h\|_{L^2}\le1$, we have
\begin{equation}\label{Tsudecomp}
\begin{aligned}
&\left|\|T_tu\|^4_{ L_{t,x}^4} -\|T_tv\|_{ L_{t,x}^4}^4 -\|T_tw\|_{ L_{t,x}^4}^4\right| \le C \left( 1+ \|u\|^3_{L^2}+ \|v\|^3_{L^2}+ \|w\|^3_{L^2}\right)\|h\|_{L^2}+ \\
&\ \ \ \ \ \ \ +4F(v,v,w,w)+F(w,v,w,v)+F(v,w,v,w)+\\
&\ \ \ \ \ \ +2\left[F(v,v,v,w)+F(v,v,w,v)+F(v,w,w,w)+F(w,v,w,w)\right].
\end{aligned}
\end{equation}

\end{lemma}
\begin{proof} For $u,v,w,h \in L^2(\mathbb{T})$, we have
$$\|T_tu\|_{ L_{t,x}^4}^4 -\|T_tv\|_{ L_{t,x}^4}^4 -\|T_tw\|_{ L_{t,x}^4}^4 =\int_0^B\int_\mathbb{T} |T_tv+T_tw +T_th|^4 - |T_tv|^4 -|T_tw|^4 \  dx\ dt.$$
By writing the integrand on the right hand side of the above equation as $$[T_tv+T_tw +T_th]^2 [\overline{T_tv}+\overline{T_tw}+\overline{T_th}]^2- |T_tv|^4 -|T_tw|^4 $$ and expanding, we obtain
$$\int_0^B\int_\mathbb{T} |T_tv+T_tw +T_th|^4 - |T_tv|^4 -|T_tw|^4 \ dx\ dt = A+ B+ R,$$
where $A$ is a finite sum of terms of the form $\int_0^B\int_\mathbb{T} T_tf_1 \ \overline{ T_tf_2}\ T_tf_3 \ \overline{ T_th}\  dx\ dt$ and $B$ is a finite sum of terms of the form $\int_0^B\int_\mathbb{T} \overline{T_tf_1}\ T_tf_2\ \overline{ T_tf_3} \ T_th\ dx\ dt$, with  $f_1,f_2,f_3\in \{u,v,w,h\}$, and $R$ consists of the seven terms involving $F$ on the right side of \eqref{Tsudecomp}.  We apply the triangle inequality, Lemma \ref{multi}, and Young's inequality to the terms in $A$ and $B$ to get the desired result.
\end{proof}

\begin{lemma}
\label{WBcont}
   The map $W_B:L^2(\mathbb T) \to \mathbb R$ is continuous.
\end{lemma}

\begin{proof} Take  $w=0$ in Lemma \ref{multi2}.
\end{proof}

\begin {lemma}\label{lm:7} There exists $C>0$ such that for all $v, w \in L^2(\mathbb{T})$, all $\delta >0$, and all integers  $ n_0$, $r_1$, and $ r_2$, if $r_2 - r_1 \ge \delta^{-1}$,  $\hat{v}(n)=0$ for $|n-n_0| > r_1$, and $\hat{w}(n)=0$ for $|n-n_0| < r_2$, then 
 
\begin{equation}
\begin{aligned}
\label{estsforFs} 
|F(v,v,w,w)| &\le (2\pi B + C \delta ^\frac12)\|\hat{v}\|_{\ell^2}^2\;\|\hat{w}\|_{\ell^2}^2, \\ 
|F(w,v,w,v)| &\le C \|\hat{v}\|_{\ell^2}^2\;\|\hat{w}\|_{\ell^2}^2\;  \delta ^\frac12,\\  
|F(v,w,v,w)|&\le C\|\hat{v}\|_{\ell^2}^2\;\|\hat{w}\|_{\ell^2}^2\;  \delta ^\frac12,\\ 
|F(v,v,v,w)|&\le C \|\hat{v}\|_{\ell^2}^3\;\|\hat{w}\|_{\ell^2}\;  \delta ^\frac12,\\  
|F(v,v,w,v)|&\le C \|\hat{v}\|_{\ell^2}^3\;\|\hat{w}\|_{\ell^2}\;  \delta ^\frac12,\\ 
|F(v,w,w,w)|&\le C \|\hat{v}\|_{\ell^2}\;\|\hat{w}\|_{\ell^2}^3\;  \delta ^\frac12,\\ 
|F(w,w,w,v)| &\le C\|\hat{v}\|_{\ell^2}\;\|\hat{w}\|_{\ell^2}^3\;  \delta ^\frac12.\\
  \end{aligned}
\end{equation} 
\end{lemma}
\begin{proof} For any $u_1, u_2, u_3, u_4 \in L^2(\mathbb T)$, we have by Fubini's theorem and Parseval's theorem that
\begin{equation}
\label{sumforF}
\begin{aligned}
F(u_1,u_2,u_3,u_4) &= 2\pi\int_0^B \sum_n \mathcal F(T_t u_1\overline{T_tu_2}T_tu_3)[n]\ \overline{\widehat{u_4}}(n) \ dt\\
&=2\pi \int_0^B\sum_{n}  e^{in^2t}\left(\widehat{T_tu_1}*\widehat{\overline{T_tu_2}}*\widehat{T_tu_3}\right)[n]\ \overline{\widehat{u_4}}(n)  \ dt\\
&=2\pi \int_0^B\sum_{n} \sum_{n_1} \sum_{n_2}  e^{in^2t}\ \widehat{T_tu_1}(n-n_1-n_2)\ \widehat{\overline{T_tu_2}}(n_1)\ \widehat{T_tu_3}(n_2)\ \overline{\widehat{u_4}}(n)\ dt\\
&=2\pi \sum_{n_1} \sum_{n_2} \sum_{n_3}\ \widehat{u_1}(n_3)
\ \widehat{\overline{u_2}}(n_1)\ \widehat{u_3}(n_2)\ \overline{\widehat{u_4}}(n_1 + n_2 + n_3)\int_0^B e^{-2it(n_1+ n_3)(n_1 + n_2)} \ dt.
\end{aligned}
\end{equation} 
where all of the sums are taken over $\mathbb Z$, and in the last expression we used a new index of summation $n_3 = n-n_1-n_2$. 
Taking  $u_1=u_2=v$ and $ u_3 = u_4 = w$ in \eqref{sumforF}, we get
\begin{equation}\label{lambda1est}
|F(v,v,w,w)|
 \le
   2\pi \sum_{n_1} \sum_{n_2} \sum_{n_3}\left|\widehat{v}(n_3)\widehat{\overline{v}}(n_1)  
    \widehat{w}(n_2)\overline{\widehat{w}}(n_1 + n_2 + n_3)\int_0^B e^{-2it(n_1+ n_3)
    (n_1 + n_2)} \ dt\right|. 
     \end{equation}
     
We write the right-hand side of \eqref{lambda1est} as the sum of four parts, $(I)+(II)+(III)+(IV)$, where $(I)$ is the sum over all terms for which $|n_1+n_3|=0$; 
$(II)$ is the sum over the terms for which $|n_1+n_2|=0<|n_1+n_3|$;
$(III)$ is the sum over the terms for which $|n_1+n_2|\ge|n_1+n_3|\ge1$; 
and $(IV)$ is the sum over the terms for which $|n_1+n_3| > |n_1+n_2|\ge 1$.

 Then we have
\begin{equation}\label{II}
(I)=2\pi B \sum_{n_2} \sum_{n_3}\left|\widehat{v}(n_3)\overline{\widehat{v}}(n_3)\widehat{w}(n_2)\overline{\widehat{w}}(n_2)\right|
=2\pi B \|\widehat{v}\|_{\ell^2}^2\,\|\widehat{w}\|_{\ell^2}^2.
\end{equation}

For  $(II)$, we have
\begin{equation}
 (II) \le 2\pi B \sum_{n_2} \sum_{n_3}\left|\widehat{v}(n_3)\widehat{\overline{v}}(-n_2)\widehat{w}(n_2)\overline{\widehat{w}}(n_3)\right|
  =2\pi B \sum_{n_2}\sum_{n_3}\left|\widehat{v}(n_3)\widehat{v}(n_2)\widehat{w}(n_2)\widehat{w}(n_3)\right|
=0,
\label{I} 
\end{equation}
because the assumptions on the supports of $v$ and $w$ in Lemma \ref{lm:7}  imply that  $\widehat{v}(n_2)\widehat{w}(n_2)=0$ for all $n_2 \in \mathbb Z$.

Before obtaining estimates for $(III)$ and $(IV)$ we note that,
in light of the assumptions on the supports of $\widehat v$ and $\widehat w$, for  
 $\widehat{v}(n_3)\ \widehat{\overline{v}}(n_1)\ \widehat{w}(n_2)\ \overline{\widehat{w}}(n_1+n_2+n_3)$ to be nonzero
 we must have  $|n_3 - n_0|\le r_1$,  $|n_1 + n_0|\le r_1$, $|n_2 - n_0|\ge r_2$, 
and $|n_1+ n_2+n_3 - n_0|\ge r_2$.

To estimate $(III)$, we first observe that if $|n_1+n_2|\ge|n_1+n_3|\ge1,$ then in all nonzero terms of the sum,
 $$
 \begin{aligned}
 1+|n_1 +n_2|\,|n_1+n_3|&\ge|n_1 +n_2|\,|n_1+n_3|\\
&\ge (|n_1 +  n_2 + n_3 - n_0| - | n_3- n_0 |)^{\frac12}\,|n_1+n_3|^{\frac32} \\
&\ge (r_2- r_1)^{\frac12}\,|n_1+n_3|^{\frac32} \\
&\ge \delta^{-1/2}|n_1+n_3|^{\frac32}.
\end{aligned}
$$
Define $K_1(n)=\chi_{|n|\ge 1}\,\, |n|^{-\frac32}$, so that $\|K_1\|_{\ell^1} < \infty$, and define $K_2(n)= K_1(n)\,[(|\hat{w}(-.)|*|\overline{\widehat{w}}| )(n)]$. Using \eqref{integralest}, we can write
\begin{equation}
\label{III}
\begin{aligned}
(III)&\le C\sum_{n_1} \sum_{n_2} \sum_{n_3}\left[\frac{\chi _{|n_1+n_2|\ge|n_1+n_3|\ge1}}{1+|n_1 +n_2||n_1 + n_3|}\right] \left|\widehat{v}(n_3)\widehat{\overline{v}}(n_1)\widehat{w}(n_2)\overline{\widehat{w}}(n_1 + n_2+ n_3)\right|\\
&\le C\delta^{\frac12} \sum_{n_1} \sum_{n_3}\chi _{|n_1+n_3|\ge1}\,\,|n_1 + n_3|^{-\frac32}\left||\widehat{v}(n_3)\widehat{\overline{v}}(n_1)\right| \sum_{n_2} \left|\widehat{w}(n_2)\overline{\widehat{w}}(n_1 + n_2+ n_3)\right|\\
&= C \delta^{\frac12} \sum_{n_1} \sum_{n_3} K_1(n_1+n_3)\left|\widehat{v}(n_3)\widehat{\overline{v}}(n_1) \right|\left( \big|\widehat{w}(-.)\big| *\big|\overline{\widehat{w}}\big|\right)(n_1+n_3) \\
&= C \delta^{\frac12} \sum_{n_1} \sum_{n_3} K_2(n_1+n_3) \, \big|\widehat{v}(n_3)\big|\, \big|\widehat{\overline{v}}(n_1)\big|\\
&=C \delta^{\frac12} \sum_{n_1} \big|\overline{\widehat{v}}(-n_1)\big|\,\,  (K_2 *\big|\widehat{v}(-.)\big|)(n_1)\\
&\le  C \delta^{\frac12} \left\|\widehat{v}\right\|_{\ell^2}\left\| K_2 *\big|\widehat{v}(-.)\big|\right\|_{\ell^2}\\
&\le C \delta^{\frac12} \|\widehat{v}\|^2_{\ell^2}  \| K_2\|_{\ell^1}\\
&\le C \delta^{\frac12} \|\widehat{v}\|^2_{\ell^2} \| K_1\|_{\ell^1}  \left\| \hat{w}\right\|^2_{\ell^2}\\
&\le  C \delta^{\frac12} \| \hat{v}\|^2_{\ell^2}\|\hat{w}\|^2_{\ell^2},
\end{aligned}
\end{equation}
where Young's inequality was used in the last few estimates.  

To estimate $(IV)$, we observe that if $|n_1+n_3|>|n_1+n_2|\ge1$, then in all nonzero terms of the sum,
$$
\begin{aligned}
1+|n_1 +n_2|\,|n_1+n_3|&\ge|n_1+n_2|^2\\
&=|(n_1 +  n_2 + n_3 - n_0) - ( n_3- n_0 )|^{\frac12}\,|n_1+n_2|^{\frac32} \\
&\ge(2\delta^{-1})^{\frac12}\,|n_1+n_2|^{\frac32}.
\end{aligned}
$$
This time we let $ K_3= K_1(n)(|\hat{v}(-.)|*|\overline{\widehat{w}}| )(n)$ with $K_1$ as previously defined, and we follow a similar argument as the one used to estimate $(III)$ to obtain
 \begin{equation}
 \label{IV}
 \begin{aligned}
 (IV)&\le2\pi C \sum_{n_1} \sum_{n_2} \sum_{n_3}\left[\frac{\chi _{|n_1+n_3|\ge|n_1+n_2|\ge1}}{1+|n_1 +n_2||n_1 + n_3|}\right] \left|\widehat{v}(n_3)\widehat{\overline{v}}(n_1)\widehat{w}(n_2)\overline{\widehat{w}}(n_1 + n_2+ n_3)\right|\\
&\le  C  \delta^{\frac12} \| \hat{v}\|^2_{\ell^2}\|\hat{w}\|^2_{\ell^2}.
\end{aligned}
\end{equation}
Taking  the sum of the estimates in  \eqref {II}, \eqref{I}, \eqref {III} and \eqref{IV} now gives the desired estimate for $F(v,v,w,w)$.

To  illustrate the proofs of the remaining estimates in \eqref{estsforFs}, consider for example the estimate for $F(v,v,w,v)$.   Taking $u_1=u_2=u_4=v$ and $ u_3 = w$ in \eqref{sumforF}, we get
\begin{equation}
\begin{aligned}
\label{lambda5est}
\left|F(v,v,w,v)\right| &\le
2\pi \sum_{n_1} \sum_{n_2} \sum_{n_3}\left|\widehat{v}(n_3)\widehat{\overline{v}}(n_1)\widehat{w}(n_2)\overline{\widehat{v}}(n_1 + n_2 + n_3)\int_0^B e^{-2it(n_1+ n_3)(n_1 + n_2)} \ dt\right|\\
 &= (I)+ (II)+(III)+(IV),
 \end{aligned}
 \end{equation}
  where  $(I)$ to $(IV)$ are defined in the same way as in the paragraph following \eqref{lambda1est}.  
 The sums $(III)$ and $(IV)$ in \eqref{lambda5est} can be estimated in the same way as the analogous sums in \eqref{lambda1est},
 and the same argument used to prove \eqref{I} shows that $(II)=0$ here as well.  In contrast to \eqref{II}, however, here we find that
$(I)=0$.   Indeed, we can write
$$ (I)=2\pi B \sum_{n_2} \sum_{n_3}\left|\widehat{v}(n_3)\overline{\widehat{v}}(n_3)\widehat{w}(n_2)\overline{\widehat{v}}(n_2)\right|.$$
Because of our assumptions on the supports of $\widehat v$ and $\widehat w$, we have
$\widehat v(n_2)\widehat w(n_2)=0$ for all $n_2 \in \mathbb Z$; and therefore $(I)=0$.   It follows that the desired estimate holds for $F(v,v,w,v)$.

The proofs of the remaining estimates in  \eqref{estsforFs} proceed in the same way as the proof of the estimate for $F(v,v,w,v)$.
\end{proof}

\begin{lemma}\label{Dconvg}
Let  $\{u_j\}_{j\in \mathbb N}\subset L^2(\mathbb{T})$ be a sequence such that $\|u_j\|^2_{L^2} = 1$ for all $j \in \mathbb N$.  Suppose that the sequence $\{\widehat{u_j}\}$ in $\ell^2(\mathbb Z)$  vanishes in the sense of Lemma \ref{conccomp}. Then $D_B(u_j)\to 0$ as $j \to \infty.$
\end{lemma}

\begin{proof}  Let $\beta \ge1.$ From the proof of  Lemma \ref{Tsu}, we see that there exists $C$ depending only on $B$ such that
$$
|D_B(\widehat{u_j})|\le C \sum_{l\neq0} \sum_{n}\sum_{p\neq 0} \frac 1{1+ |lp|}\left|\widehat{u_j} (n)\widehat{u_j}(n-l)\widehat{u_j}(n-p)\widehat {u_j}(n-p-l)\right|\\
$$
We can decompose the triple sum on the right-hand side into three parts, writing it as $(I)+ (II)+(III)$,  where $(I)$ is the sum over all $(l,n,p)$ such that $1\le |l|\le \beta$ and $1\le |p|\le \beta$,  
$(II)$ is the sum over all $(l,n,p)$  such that $|p|>|l|\ge 1$ and $|p|>\beta$, and $(III)$ is the sum over all $(l,n,p)$ such that $|l|>\beta$ and $|l| \ge |p|\ge 1$.

 To estimate $(I)$, we write 
 $$ 
 \begin{aligned}
(I)  &\le   \sum_n|\widehat{u_j} (n)|\sum_{l=-\beta}^\beta\left| \widehat{u_j}(n-l)\right| \sum^{\beta}_{p=-\beta}\left|\widehat{u_j}(n-p)\widehat {u_j}(n-p-l)\right|\\
 &\le   \sum_n|\widehat{u_j} (n)|\sum_{l=-\beta}^\beta\left| \widehat{u_j}(n-l)\right| \left(\sum^{\beta}_{p=-\beta}\left|\widehat{u_j}(n-p)\right|^2\right)^{1/2}\|\widehat{u_j}\|_{\ell^2}\\
 &\le \|\widehat{u_j}\|_{\ell^2} \sup_{m \in \mathbb Z} \left(\sum^{m+\beta}_{r=m-\beta}|{\widehat {u_j}}(r)|^2\right)^{1/2} \sum_n| \widehat{u_j} (n)|\sum_{l=-\beta}^\beta|\widehat{u_j}(n-l)|  \\
&\le \|\widehat{u_j}\|_{\ell^2} \sup_{m \in \mathbb Z} \left(\sum^{m+\beta}_{r=m-\beta}|{\widehat {u_j}}(r)|^2\right)^{1/2}\,\|\widehat{u_j}\|_{\ell^2} \|\chi_{[-\beta, \beta]}*|\widehat{u_j}|\|_{\ell^2} \\
&\le 2 \beta \|\widehat{u_j}\|^3_{\ell^2} \sup_{m \in \mathbb Z} \left(\sum^{m+\beta}_{r=m-\beta}|\widehat {u_j}(r)|^2\right)^{1/2}.
\end{aligned}$$

To estimate $(II)$ we observe that for all $(l,n,p)$ which appear in that sum,
$$
1+|lp|> |l|^\frac32 |p|^\frac12> \beta^{\frac12}|l|^\frac32.
$$
We write
$$
\begin{aligned}
(II)&= \sum_l \sum_n \sum_p\frac{\chi _{\{|p|>|l|\ge1\} }(l,p)}{1+ |lp|} |\widehat{u_j} (n)\widehat{u_j}(n-l)\widehat{u_j}(n-p)\widehat {u_j}(n-p-l)|\\
 &\le \beta^{-\frac12} \sum_n |\widehat{u_j} (n)| \sum_l  \chi _{|l|\ge1}(l)|l|^{-\frac32} |\widehat {u_j}(n-l)|\sum_p|\widehat{u_j}(n-p)\widehat{u_j}(n-p-l)|\\
  &\le \beta^{-\frac12}\|\widehat{u_j}\|_{\ell^2}^2 \sum_n |\widehat{u_j} (n)| \sum_l  \chi _{|l|\ge1}(l)|l|^{-\frac32}|\widehat {u_j}(n-l)| .\\
\end{aligned}
$$
Therefore, if we define $K_1(l)=\chi _{|l|\ge1}(l)|l|^{-\frac32}$ and apply Young's convolution inequality, we obtain that 
$$
 (II)\le   \beta^{-\frac12} \| \widehat{u_j}\|^3_{\ell^2}  \left\| K_1 *\left|\widehat{u_j}\right|\right\|_{\ell^2}
 \le \beta^{-\frac12} \left\|\widehat{u_j}\right\|^3_{\ell^2}  \| K_1\|_{\ell^1}\left\| \widehat{u_j}\right\|_{\ell^2}
\le  C\beta^{-\frac12} \left\| \widehat{u_j}\right\|^4_{\ell^2}.
$$
 
In estimating $(III)$ we can use that
$$
1+|lp|> |l|^\frac12|p|^\frac32 > \beta^{\frac12}|p|^\frac32.
$$
We write
$$
\begin{aligned}(III)&=\sum_l \sum_n \sum_p\frac{\chi _{\{|l|\ge|p|\ge1\}}(l,p)}{1+ |lp|} |\widehat{u_j} (n)\widehat{u_j}(n-l)\widehat{u_j}(n-p)\widehat {u_j}(n-p-l)|\\
&\le \beta^{-\frac12}\sum_n |\widehat{u_j} (n)|    \sum_p\chi _{\{|p|\ge 1\}}(p)|p|^{-\frac32}| \widehat{u_j}(n-p)|\sum_l |\widehat{u_j}(n-l)\widehat {u_j}(n-p-l)|\\
&\le \beta^{-\frac12}\|\widehat{u_j}\|_{\ell^2}^2 \sum_n |\widehat{u_j} (n)| \sum_p  \chi _{\{|p|\ge 1\}}(p)|p|^{-\frac32} |\widehat {u_j}(n-p)| \\
&=\beta^{-\frac12}\|\widehat{u_j}\|_{\ell^2}^2 \sum_n |\widehat{u_j} (n)| (K_1 \ast|\widehat{u_j}|)(n)\\ 
 &\le C \beta^{-\frac12} \left\| \widehat{u_j}\right\|^4_{\ell^2},
\end{aligned}$$
where again we applied Young's convolution inequality in the last step.
 
Combining the above estimates for $(I)$, $(II)$, and $(III)$, we obtain that
$$
|D_B(\widehat{u_j})| \le 
C \beta^{-\frac12} \left\| \widehat{u_j}\right\|^4_{\ell^2}  
+ C \beta \|\widehat{u_j}\|^3_{\ell^2} \sup_{m \in \mathbb Z} \left(\sum^{m+\beta}_{r=m-\beta}|\widehat {u_j}(r)|^2\right)^{1/2}.
$$
Now, if $\{\widehat{u_j}\}$ vanishes, then for each fixed $\beta \ge 1$ and $\epsilon >0$ there exists $N\in \mathbb{N}$ such that for all $j\ge N$,
$$ \sup_{m \in \mathbb Z}\sum^{m+\beta}_{r=m-\beta}|\widehat {u_j}(r)|^2 < \epsilon^6.$$
In particular, for $\beta = \epsilon^{-2}$, there exists $N$ such that for all $j\ge N$, 
$$
|D_B(\widehat{u_j})|\le C \epsilon^{1/2} \left\| \widehat{u_j}\right\|^4_{\ell^2}  + C \epsilon^{-2} {(\epsilon^6)}^{\frac12}\|\widehat{u_j}\|^3_{\ell^2}     \le C \epsilon.$$
This shows that  
$\lim _{j\to \infty} D_B(\widehat{u_j}) = 0$.
\end{proof}

\section {Proof of Theorem \ref{mainthm}}
\label{sec:proofmainthm}

We first prove part (i) of the Theorem, which is the main part.

Fix $B>0$, and suppose $J_{B,1}> B/\pi$.   Let $\{u_j\}_{j \in \mathbb N}$ be a maximizing sequence in $L^2(\mathbb T)$ for $J_{B,1}$, so that $\|u_j\|_{L^2}=1$ for all $j \in \mathbb N$ and $\lim_{j \to \infty}W_B(u_j)=J_{B,1}$.  We have that $\|\widehat{u_j}\|_{\ell^2}^2=1/(2\pi)$ for all $j \in \mathbb N$, so Lemma \ref{conccomp} applies with $M=1/(2\pi)$, and asserts that there are three types of behavior that the sequence $\{\widehat{u_j}\}$ could exhibit.   We claim that in the present situation, vanishing and splitting do not occur, so that only tightness is possible. 

We suppose first, for the sake of contradiction, that the sequence $\{\widehat{u_j}\}$  is vanishing. Then from \eqref{WandD} we have that
\begin{equation}
\label{wuj}
W(u_j) = 4\pi B \|\widehat{u_j}\|^4_{\ell^2}- 2\pi B \|\widehat{u_j}\|^4_{\ell^4}+ D_B(u_j)
\end{equation}
for all $j \in \mathbb N$.    On the other hand, from Lemmas \ref{l4convg} and \ref{Dconvg} we have that $\|\widehat{u_j}\|_{\ell^4} \to 0$ and $D_B(u_j) \to 0$ as $j \to \infty$.
Therefore, taking $j \to \infty$ in \eqref{wuj}, we get that $J_{B,1} = B/\pi$, contradicting the assumption that $J_{B,1} > B/\pi$.   Thus $\{\widehat{u_j}\}$ cannot vanish. 

Next suppose, again for contradiction, that $\{\widehat{u_j}\}$ exhibits splitting.  
Let $a_j(n)=\widehat {u_j}(n)$ for $n \in \mathbb N$, fix $\delta >0$, and for this $\delta$ define $\alpha \in (0,1/(2\pi))$ and for each $j \in \mathbb N$ define sequences $\{b_j(n)\}_{n \in \mathbb N}$ and $\{c_j(n)\}_{n \in \mathbb N}$ as in alternative 2 of Lemma \ref{conccomp}.   For each $j \in \mathbb N$, let $v_j, w_j \in L^2(\mathbb T)$ be such that
$\widehat{v_j}(n)=b_j(n)$ and $\widehat{w_j}(n)=c_j(n)$ for all $n \in \mathbb N$.  From Lemmas \ref{multi2} and \ref{lm:7} we have that, for all $j \in \mathbb N$,
$$
\left|W_B\left(u_j\right) - W_B\left(v_j\right) - W_B\left(w_j\right)\right| \le 8\pi B\|\hat{v}_j\|_{\ell^2}^2\|\hat{w}_j\|_{\ell^2}^2 
+C\delta ^\frac12,
 $$
 where $C$ is independent of $\delta$ and $j$.  Therefore
 $$
 \begin{aligned}
 W(u_j)&\le W(v_j)+W(w_j) + 8\pi B\|\widehat{v_j}\|_{\ell^2}^2\;\|\widehat{w_j}\|_{\ell^2}^2+C \delta^{\frac12}\\
&\le \|v_j\|^4_{L^2} J_{B,1} +\|w_j\|^4_{L^2} J_{B,1}+8\pi B \|\widehat{v_j}\|_{\ell^2}^2\|\widehat{w_j}\|_{\ell^2}^2 +C \delta^{\frac12}\\
&=4\pi^2\|\widehat{v_j}\|_{\ell^2}^4 J_{B,1} +4\pi^2\|\widehat{w_j}\|_{\ell^2}^4 J_{B,1}) + 8\pi B \|\widehat{v_j}\|_{\ell^2}^2 \|\widehat{w_j}\|_{\ell^2}^2+  C \delta^{\frac12}.
\end{aligned}
$$
Recalling that $\|\widehat{v_j}\|_{\ell^2}^2 \le \alpha + \delta$ and $\|\widehat{w_j}\|_{\ell^2}^2 \le (M-\alpha)+\delta = (1/(2\pi)-\alpha)+\delta$, we obtain that
\begin{equation}
\label{wujlesomething}
\begin{aligned}
W(u_j)
&\le 4\pi^2\left(\alpha^2 +\left(\frac{1}{2\pi} - \alpha\right)^2\right) J_{B,1}+ 8\pi B \alpha\left(\frac{1}{2\pi}- \alpha\right)+C\delta^\frac12 +C\delta+C\delta^2.\\  
\end{aligned}
\end{equation}
 
Taking the limit as $j\to \infty$ followed by the limit as $\delta\to 0$ in \eqref{wujlesomething} results in 
\begin{equation}
\label{ineqjb}
J_{B,1}\le 4\pi^2\left(\alpha^2 +\left(\frac{1}{2\pi} - \alpha\right)^2\right) J_{B,1}+ 8\pi B \alpha\left(\frac{1}{2\pi} - \alpha\right).
\end{equation}
But since $0 < \alpha < \frac{1}{2\pi}$, the inequality \eqref{ineqjb} implies that $J_{B,1} \le B/\pi$, again contradicting the assumption that $J_{B,1} > B/\pi$.   Therefore $\{\widehat{u_j}\}$ can not split, either.

By Lemma \ref{conccomp}, the only remaining possibility for $\{\widehat{u_j}\}$ is that one of its subsequences, when suitably translated, is tight.  In other words, denoting this subsequence again by $\{\widehat{u_j}\}$, we  can  assert the existence of integers $m_1,m_2, m_3,...$ such that for each $\epsilon>0$, there exists an integer $ r=r(\epsilon)>0$ with the property that 
\begin{equation}
\label{ujtight}
\sum_{n=m_j -r}^{m_j +r} |\widehat{u_j}(n)|^2 \ge \frac{1}{2\pi} -\epsilon
\end{equation}
for all  $j\in \mathbb{N}$.

Define $v_j(x) = e^{-im_jx} u_j(x)$ for $j \in \mathbb N$, so that $\hat v_j(n) = \hat u_j(n+m_j)$ for all $n \in \mathbb Z$.   By Lemma \ref{invariant} $\{v_j\}$ is also a maximizing sequence for $J_{B,1}$.
Also, from \eqref{ujtight} we have that for each $\epsilon>0,$ there exists an integer $ r>0$ with the property that for all $ j\in \mathbb{N},$
$$\sum_{n=-r}^{r} |\widehat{v_j}(n)|^2 \ge \frac{1}{2\pi} -\epsilon.$$

Since the sequence $\{v_j\}_{j \in \mathbb N}$ is bounded in $L^2(\mathbb T)$, with $\|v_j\|_{L^2}=1$ for all $j$, there exists a subsequence, still denoted by $\{v_j\},$ that converges weakly to some function $u_0\in L^2(\mathbb T)$ with $\|u_0\|_{L^2} \le 1$. 

We claim that in fact $\|u_0\|_{L^2} = 1$.   To prove this, we start by fixing an arbitrary $k \in \mathbb N$.  Let $\epsilon_k= \frac1k$  and  choose $r_k=r(\epsilon_k)= r(\frac1k)$.  We define 
 $\mu_k: \mathbb Z \to \{0,1\}$ by setting  $\mu_k(n)=1$ for $|n| \le r_k$ and $\mu_r(n)=0$ for $|n| > r_k$; and then define the low- and high-frequency components $v^{(l)}_{j,k}$ and $v^{(h)}_{j,k}$  of $v_j$ by setting
  $$
  \mathcal F\left(v^{(l)}_{j,k}\right)[n] = \mu_k(n)\widehat{v_j}(n)
  $$ 
  and
  $$
  \mathcal F\left(v^{(h)}_{j,k}\right)[n] = \left(1-\mu_k(n)\right)\widehat{v_j}(n)
  $$  
  for all $n \in \mathbb Z$.  
  
  We then have
\begin{equation} 
\label{lowestimate}
\|v^{(l)}_{j,k}\|^2_{ H^1}= \sum_n(1+ |n|^2)|\mu_k(n)\widehat{v}_j(n)|^2
\le (1+4r_k^2)\| \widehat{v_j}\|_{\ell^2}^2= \frac{1}{2\pi}\left(1+4r_k^2\right)
 \end{equation} 
 and
\begin{equation}\label{highestimate} 
\|v^{(h)}_{j,k}\|^2_{L^2} =2\pi\sum_{|n|> r_k}|\widehat{v}_j(n)|^2
=2\pi\left(\frac{1}{2\pi}- \sum_{|n|\le r_k}|\widehat{v}_j(n)|^2\right)
\le2\pi\epsilon_k. 
\end{equation}

   Siince \eqref{lowestimate} bounds $\{v^{(l)}_{j,k}\}$ in $H^1$ norm and \eqref{highestimate} bounds $\{v^{(h)}_{j,k}\}$ in $L^2$ norm, we can assume (by passing to subsequences if necessary) that $\{v^{(l)}_{j,k}\}_{j \in \mathbb N}$ converges weakly in $H^1(\mathbb T)$ to some limit $u^{(l)}_k \in H^1(\mathbb T)$, and $\{v^{(h)}_{j,k}\}_{j \in \mathbb N}$ converges weakly in $L^2$ to some limit $u^{(h)}_k \in L^2(\mathbb T)$ with $\|u^{(h)}_k\|_{L^2} \le \sqrt{2\pi \epsilon_k}$.   We must then have $u_0 = u^{(l)}_k + u^{(h)}_k$.

By Rellich's Lemma, the inclusion of $H^1(\mathbb T)$ into $L^2(\mathbb T)$ is compact.   Therefore, again by passing to a subsequence, we can assume that $\{v^{(l)}_{j,k}\}_{j \in \mathbb N}$ converges strongly in $L^2(\mathbb T)$ to $u^{(l)}_k$.   Hence
$$ 
\begin{aligned}\|u_0\|_{L^2 (\mathbb{T})} &=\|u^{(l)}_k +u^{(h)}_k\|_{L^2 (\mathbb{T})}\\
&\ge\|u^{(l)}_k\|_{L^2 (\mathbb{T})} -\|u^{(h)}_k\|_{L^2 (\mathbb{T})}\\ 
&\ge \lim_{j \to \infty}\|v^{(l)}_{j,k}\|_{L^2 (\mathbb{T})} -\sqrt{2\pi\epsilon_k}\\  
&\ge \liminf_{j \to \infty}\left[\|v_{j}\|_{L^2 (\mathbb{T})}- \|v^{(h)}_{j,k}\|_{L^2 (\mathbb{T})}\right] -\sqrt{2\pi\epsilon_k}\\
&\ge \lim_{j \to \infty}\|v_j\|_{L^2 (\mathbb{T})} -2\sqrt{2\pi\epsilon_k}\\
&=1-2\sqrt{2\pi\epsilon_k}.
\end{aligned}
$$
 
We have thus proved that $\|u_0\|_{L^2} \ge 1 - 2\sqrt{2\pi \epsilon_k}$ for every $k \in \mathbb N$, and so we have shown that $\|u_0\|_{L^2} = 1 = \lim_{j \to \infty} \|v_j\|_{L^2}$.   This is enough to conclude that $\{v_j\}$ converges to $u_0$ not only weakly, but also in the norm of $L^2(\mathbb T)$.  Since, as noted in Lemma \ref{WBcont}, the map $W_B$ is continuous on $L^2(\mathbb T)$, it follows that $u_0$ is a maximizer for $J_{B,1}$.  This completes the proof of part (i) of Theorem \ref{mainthm}.

To prove part (ii) of the Theorem, let $\{u_j\}_{j \in \mathbb N}$ be any sequence such that $\|u_j\|_{L^2(\mathbb T)}=\sqrt{2\pi}\|\widehat{u_j}\|_{\ell^2}=1$ for all $j \in \mathbb N$ and $\{\widehat{ u_j}\}$ vanishes, in the sense of Lemma \ref{conccomp}.   For example, we could define $u_j$ by requiring that
$$
\widehat{u_j}(n)= \begin{cases} \frac{1}{\sqrt{2\pi(2j+1)}} \ &\text{for $|j| \le n$}\\
                                0 &\text{for $|j| > n$}.
                                \end{cases}
$$ 
Since $\{\widehat{u_j}\}$ vanishes, it follows from Lemmas \ref{l4convg} and \ref{Dconvg} and equation \eqref{WandD} that 
\begin{equation}
\label{WBeqBpi}
\lim_{j \to \infty} W_B(u_j)=B/\pi,
\end{equation}
 and therefore we must have $J_{B,1} \ge B/\pi$.

For part (iii) of the Theorem, assume that $J_{B,1}=B/\pi$, and take $\{u_j\}$ to be any sequence such that $\|u_j\|_{L^2(\mathbb T)}=1$ for all $j \in \mathbb N$ and $\{\widehat{ u_j}\}$ vanishes.  As in the preceding paragraph, we have that \eqref{WBeqBpi} holds, which means that $\{u_j\}$ is a maximizing sequence.  However, since $\{\widehat {u_j}\}$ vanishes, then by the remark made above in the paragraph following Lemma \ref{conccomp}, it is impossible for there to exist a subsequence $\{u_{j_k}\}$ and a sequence of integers $\{m_k\}$ such that $\{\widehat {u_{j_k}}(\cdot - m_k)\}$ converges strongly in $\ell^2(\mathbb Z)$.
This then proves part (iii).

\section{Existence of maximizers}
\label{sec:existence}

In this section we give results on the set of values of $B>0$ for which maximizers for $J_{B,1}$ exist in $L^2(\mathbb T)$.  

For what follows, it will be useful to define the map $A_B:L^2(\mathbb T) \to \mathbb R$ by
$$
A_B(u)=\frac{D_B(u)}{2\pi B}-\|\hat u\|_{\ell^4}^4,
$$
where $D_B$ is defined in \eqref{defDB}.  
We have the following corollary of Theorem \ref{mainthm}.

\begin{corollary}  
\label{ABcriterion}
Let $B>0$ be given. 
\item{(i)}  Suppose  there exists some $w \in L^2(\mathbb T)$ such that $A_B(w) > 0$.    Then $J_{B,1} > B/\pi$, and there exists a maximizer for $J_{B,1}$ in $L^2(\mathbb T)$.
\item{(ii)}  If, on the other hand, one has that $A_B(u)<0$ for all  $u \in L^2(\mathbb T)$, then $J_{B,1} = B/\pi$, and there do not exist any maximizers for $J_{B,1}$ in $L^2(\mathbb T)$.
\end{corollary}
 
\begin{proof} By Theorem \ref{mainthm}, to prove part (i) it is enough to show that $J_{B,1} > B/\pi$ holds if and only if there exists $w \in L^2(\mathbb T)$ such that $A_B(w)>0$.   Indeed, because $D(\lambda w)=\lambda^4 w$ for all $\lambda > 0$ and all $w \in L^2(\mathbb T)$, we have that $A_B(w)>0$ for some $w \in L^2(\mathbb T)$ if and only if $A_B(w)>0$ for some $w \in L^2(\mathbb T)$ with $\|w\|_{L^2}=1$. By \eqref{WandD}, this  is equivalent to saying that $W_B(w)> B/\pi$ for some $w$ with $\|w\|_{L^2}=1$.   This in turn is clearly equivalent to the assertion that  $J_{B,1} >B/\pi$.

To prove part (ii), note that if $A_B(u)<0$ for all $u \in L^2(\mathbb T)$, then from \eqref{WandD} it follows that $W_B(u) < B/\pi$ for all $u \in L^2(\mathbb T)$ such that $\|u\|_{L^2}=1$.   In particular, $J_{B,1} \le B/\pi$.   On the other hand, from part (ii) of Theorem \ref{mainthm} we have that $J_{B,1} \ge B/\pi$.   Therefore, we must have $J_{B,1} = B/\pi$, and moreover there cannot exist any $u_0 \in L^2(\mathbb T)$ such that $\|u_0\|_{L^2}=1$ and $W_B(u_0)=J_{B,1}$.
\end{proof}

For $u \in L^2(\mathbb T)$ and $p, l \in \mathbb Z$, define
\begin{equation}
\label{defapl}
a_{p,l}(u)=\sum_{n \in \mathbb N} \hat u (n)\bar{\hat u}(n-l) \bar{\hat u}(n-p)\hat{u}(n-p-l)
\end{equation}
and 
\begin{equation}
\label{defbpl}
b_{p,l}=\frac{1}{B}\int_0^B e^{-2ilpt}\ dt,
\end{equation}
so that from \eqref{defDB} we have
\begin{equation}
D_B(u)=2\pi B\sum_{l \ne 0} \sum_{p \ne 0} a_{p,l}(u) b_{p,l}.
\label{DBrewrite}
\end{equation}

\begin{lemma}  For all $B>0$ and $u \in L^2(\mathbb T)$,
\begin{equation}
A_B(u)= 4 \Re \left( \sum_{p=1}^\infty a_{p,p}(u) b_{p,p} + 2 \sum_{p=1}^\infty \sum_{l=1}^{p-1} a_{p,l}(u) b_{p,l}\right)-a_{0,0}(u),
\label{DBsimp}
\end{equation}
where $\Re z$ denotes the real part of the complex number $z$.

In particular, if the Fourier coefficients $\hat u(n)$  are real-valued for all $n \in \mathbb Z$, we have  
\begin{equation}
A_B(u)=4\sum_{p=1}^\infty a_{p,p}(u) \frac{\sin(2p^2B)}{2p^2 B} + 8 \sum_{p=2}^\infty\sum_{l=1}^{p-1} a_{p,l}(u) \frac{\sin(2plB)}{2pl B} - a_{0,0}(u).
\label{DBrealsimp}
\end{equation}
\end{lemma}

\begin{proof}
It is easy to see from \eqref{defapl} and \eqref{defbpl} that for all $u \in L^2(\mathbb T)$ and  all $p$ and $l$ in $\mathbb Z$, we have
$$
a_{p,l}(u)=a_{l,p}(u)=a_{p,-l}(u)=\overline{a_{-p,l}(u)}
$$
and
$$
b_{p,l}=b_{l,p}=b_{-p,-l}=\overline{b_{-p,l}}.
$$
In view of these identities, the statements in the Lemma follow from \eqref{DBrewrite} and the fact that
$$
a_{0,0}(u)=\|\hat u\|_{\ell^4}^4. 
$$
\end{proof}

An immediate consequence is the following nonexistence result.

\begin{corollary} 
\label{nonexistence}
If $B=N\pi$ for $N \in \mathbb N$, then $J_{B,1}= N$, and there do not exist any maximizers for $J_{B,1}$ in $L^2(\mathbb T)$.
\end{corollary}

\begin{proof}
From \eqref{defbpl} we see that if $B>0$ is an integer multiple of $\pi$, then $b_{l,p}=0$ for all integers $l$ and $p$ such that $lp \ne 0$.   Therefore, by \eqref{DBsimp},  
$$
A_B(u)=-a_{0,0}(u)=-\|\hat u\|_{\ell^4}^4 < 0
$$
for all $u \in L^2(\mathbb T)$.  The result then follows from part (ii) of Corollary \ref{ABcriterion}.
\end{proof}

To obtain existence results, we consider different test functions for $w$.   First, define $w_1 \in L^2(\mathbb T)$ by setting
$$
\widehat{w_1}(n)=\begin{cases} &1 \quad\text{for $n=0$}\\
                                                 &r \quad\text{for $n = \pm 1$}\\
                                                 &0 \quad\text{for $|n| \ge 2$},
                                                 \end{cases}
$$
where $r \in \mathbb R$.  Clearly, when $w=w_1$ we have that $a_{1,1}=r^2$ and $a_{p,l}=0$ for $(p,l) \ne (1,1)$.  Therefore, from \eqref{DBrealsimp} we get that
$$
A_B(w_1)= 4 r^2 \ \frac{\sin 2B}{2B} -(1+2r^4).
$$
Since the function $\displaystyle f(r)=\frac{1+2r^4}{4r^2}$ has a minimum value of $\frac{\sqrt 2}{2}$  at $r=2^{-\frac14}$, then there will exist a choice of $r \in \mathbb R$ for which $A_B(w_1)>0$, provided that
\begin{equation}
\frac{\sin 2B}{2B} < \frac{\sqrt 2}{2}.
\end{equation}
Thus we see that there exists $w_1 \in L^2(\mathbb R)$ for which $A_B(w_1)>0$, provided $B \in (0,B_0)$,
where  $B_0 \approx 0.6958$ is the positive solution of $(\sin 2B_0)/2B_0 = \sqrt 2/2$.

Next, define $w_2 \in L^2(\mathbb T)$ by setting
$$
\widehat{w_2}(n)=\begin{cases} &1 \quad\text{for $n=0$}\\
                                                 &r \quad\text{for $n = \pm 1$}\\
                                                 &s \quad\text{for $n= \pm 2$}\\
                                                 &0 \quad\text{for $|n| \ge 3$},
                                                 \end{cases}
$$
 where $r, s \in \mathbb R$.  Here we see that the only nonzero values of $a_{p,l}(w_2)$ which appear on the right-hand side of \eqref{DBrealsimp} when $u=w_2$ are 
 \begin{equation}
 \begin{aligned}
a_{0,0}(w_2)&=1+2r^4+2s^4\\
 a_{1,1}(w_2)&=r^2(1+2s)\\
 a_{2,1}(w_2)&=2r^2s\\
 a_{3,1}(w_2)&=r^2s^2\\
 a_{2,2}(w_2)&=s^2.
  \end{aligned}
  \end{equation}
  Therefore \eqref{DBrealsimp} gives
  $$
  A_B(w_2)=4r^2(1+2s)\left(\frac{\sin 2B}{2B}\right)+16r^2s\left(\frac{\sin 4B}{4B}\right)+
  8r^2s^2\left(\frac{\sin 6B}{6B}\right)+4s^2\left(\frac{\sin 8B}{8B}\right)-(1+2r^4+2s^4).
  $$ 
   Computations with Mathematica indicate that 
$\max\{A_B(w_2):  (r,s) \in \mathbb R^2\}$ is positive for all $B$ such that $0 < B < B_1$ where $B_1 = 0.919 \pm .001$.
    
  In fact,  if we define $w_3 \in L^2(\mathbb T)$ by setting
$$
\widehat{w_3}(n)=\begin{cases} &1 \quad\text{for $n=0$}\\
                                                 &p+iq\quad\text{for $n = \pm 1$}\\
                                                 &0 \quad\text{for $n \ge 2$},
                                                 \end{cases}
$$
then computations with Mathematica show that $\max\{A_B(w_3):  (p,q) \in \mathbb R^2\}$ is positive for all $B$ in the interval $0 < B < B_3$, where $B_3=1.39 \pm .01$.  For $B$ near $B_3$, the maximum occurs near $p=0.6$ and $q=0.5$. 

We can go a bit further by defining $w_4 \in L^2(\mathbb T)$ by
$$
\widehat{w_4}(n)=\begin{cases} &1 \quad\text{for $n=0$}\\
                                                 &p+iq \quad\text{for $n = \pm 1$}\\
                                                 &p+iq \quad\text{for $n= \pm 2$}\\
                                                 &0 \quad\text{for $|n| \ge 3$}.
                                                 \end{cases}
$$
Then computations with Mathematica show that   $\max\{A_B(w_4):  (p,q) \in \mathbb R^2\}$ is positive for all $B$ in the interval $0 < B < B_4$, where $B_4=2.60 \pm .01$.  For $B$ near $B_4$, the maximum is attained near $p=0.7$ and $q=0.6$.  

From these computations and Corollary \ref{ABcriterion} we then obtain the following existence result: 

\begin{corollary}
\label{existencerange}
There exist maximizers for $J_{B,1}$ in $L^2(\mathbb T)$ for all $B$ in the interval $0 < B < B_4$, where $B_4 = 2.60 \pm -.01$.
\end{corollary}
  
\section{Stability of sets of ground-state solutions of the periodic DMNLS equation}
\label{sec:stability}

As mentioned in the introduction, the periodic DMNLS equation for functions of period $L$ in $x$ takes the form
\begin{equation}
u_t = -i \nabla H_L(u)
\label{DMNLSperL}
\end{equation}
   where  
$$
H_L(u)=-\frac{2\pi}{L}\int_0^L\int_0^1 |T^L_t u(x)|^4\ dt \ dx.
$$  
The operator $T^L_t$ is defined as a Fourier multiplier operator on $L^2_{\rm per}(0,L)$ by setting
$$
\mathcal F_L(T^L_t u)[n]=e^{-i(2\pi n/L)^2 t} \mathcal F_L u[n]
$$
 for all $n \in \mathbb Z$, where $\mathcal F_L$ denotes the Fourier transform on $L^2_{\rm per}(0,L)$ (see Section \ref{sec:notation} for notation).
 In particular, from \eqref{defT} we see that $T_t=T^{2\pi}_t$. 
 
 Equation \eqref{DMNLSperL} is globally well-posed in $L^2_{\rm per}(0,L)$, in the sense that for every $u_0(x) \in L^2_{\rm per}(0,L)$, there is a unique strong solution $u(x,t)$ of  \eqref{DMNLSperL} in $L^2_{\rm per}(0,L)$  with $u(x,0) = u(x)$.   Moreover, $H_L(u)$ and $P(u):=\frac12 \int_0^L|u|^2\ dx$ are conserved quantities for such solutions.  (See \cite{adekoya} for details.)
 
 A solution of \eqref{DMNLSperL} of the form 
 \begin{equation}u(x,t)=e^{i \omega t}\phi(x),
 \label{BS}
 \end{equation}
 where $\phi \in L^2_{\rm per}(0,L)$,
  is called a bound-state solution with profile function $\phi$.   Substituting into \eqref{DMNLSperL}, we see that $\phi \in L^2_{\rm per}(0,L)$ is the profile function of a bound-state solution if and only if $\phi$ is satisfies the equation
  \begin{equation}
 \nabla H_L(\phi) = \omega \phi
 \label{EL}
 \end{equation}
 for some $\omega \in \mathbb R$.
 
Note that \eqref{EL} is the Euler-Lagrange equation for the variational problem of minimizing $H_L(u)$ subject to the constraint that $P(u)$ be held constant, with $\omega$ playing the role of the Lagrange multiplier.  Thus profile functions for bound-state solutions may be characterized as critical points of the variational problem.   If a non-zero bound-state profile $\phi$ is actually a minimizer for the variational problem, then we say that the bound-state solution is a ground-state solution. That is, a bound-state solution \eqref{BS} is a ground-state solution if $H(\phi) \le H(\psi)$ for all $\psi \in L^2_{\rm per}(0,L)$ such that $P(\psi)=P(\phi)>0$.
 
 For given $\lambda > 0$ and $L>0$, we define $S_{L,\lambda}$ to be the set of all minimizers for $H_L(u)$ subject to the constraint $P(u)=\lambda$.  (Note that it may happen that no such minimizers exist, in which case $S_{L,\lambda}$ is empty.) Thus, every element of $S_{L,\lambda}$ is a ground-state solution profile; and conversely every ground-state profile belongs to $S_{L,\lambda}$ for some $\lambda > 0$.  Because $H_L(u)$ and $P(u)$ are invariant under translations and under the action of multiplication by $e^{i\theta}$ for $\theta \in \mathbb R$, then $S_{L,\lambda}$ is also invariant under these operations.   That is, if $\phi \in S_{L,\lambda}$, then $e^{i\theta}\phi(x+x_0)$ is also in $S_{L,\lambda}$ for every $\theta \in \mathbb R$ and every $x_0 \in \mathbb R$.  Another way of putting this fact is that $S_{L,\lambda}$ is invariant under translations both in $x$ and in Fourier space.

    We say that a sequence $\{u_n\}$ in $L^2_{\rm per}(0,L)$ is a minimizing sequence for $S_{L,\lambda}$ if $P(u_n)=\lambda$ for all $n \in \mathbb N$, and $H_L(u_n) \to I_{L,\lambda}$ as $n \to \infty$, where 
 $$
 I_{L,\lambda}=\inf\left\{H_L(u): u \in L^2_{\rm per}(0,L) \quad \text{and $P(u)=\lambda$}\right\}.
 $$
 
 We observe next that profiles in $S_{L,\lambda}$ are related via dilations to the maximizers for \eqref{Strichinequal} discussed in the preceding sections.  
  For $\delta > 0$, define a dilation operator $M_\delta$ on functions $u$ with domain $\mathbb R$,  by setting
 $$
(M_\delta u)(x)= u(\delta x)
 $$ 
 for $x \in \mathbb R$.

\begin{lemma}
\label{dilation} Suppose $L>0$ and $\lambda >0$ are given, and let $\delta = L/(2\pi)$ and $B=(2\pi/L)^2$.   Then $\psi \in L^2_{\rm per}(0,L)$ is in $S_{L,\lambda}$ if and only if $M_\delta(\psi) \in L^2(\mathbb T)$ is a maximizer for $J_{B,\lambda/\delta}$.
\end{lemma}

\begin{proof}   We have that  $u \in L^2_{\rm per}(0,L)$ with $P(u)=\lambda$ if and only if  $v=M_\delta u \in L^2(\mathbb T)$ with $\|v\|_{L^2}^2= \lambda/\delta$.  A calculation shows that
 $$
 T_t(v)(x)=M_\delta\left(T^L_{\delta^2 t}(u)\right),
 $$
whence one obtains that
 $$
 H_L(u)=-\frac{1}{B}W_B(v).
 $$  
Taking the infimum over all $u \in L^2_{\rm per}(0,L)$ with $P(u)=\lambda$, or equivalently over all $v \in L^2(\mathbb T)$ with $\|v\|^2_{L^2}= \lambda/\delta$, we obtain the desired result.
\end{proof}
  
 \bigskip
 
 \begin{thm} Suppose $L > 2\pi/\sqrt{B_4}$, where $B_4$ is as defined in Corollary \ref{existencerange}. 
  Then for every $\lambda > 0$, $S_{L,\lambda}$ is nonempty, and is furthermore stable, in the following sense.   For $u \in L^2_{\rm per}(0,L)$, define
 $$
 d(u, S_{L,\lambda})= \inf_{\phi\in S_{L,\lambda}}\|u-\phi\|_{L^2_{\rm per}(0,L)}.
 $$
 Then for every $\epsilon > 0$, there exists $\delta > 0$ such that if $u_0 \in L^2_{\rm per}(0,L)$  with $d(u_0,S_{L,\lambda})<\delta$,  the solution $u(x,t)$ of \eqref{DMNLSperL} with initial data $u(\cdot,0)=u_0$ will satisfy $d(u(\cdot,t), S_{L,\lambda})< \epsilon$ for all $t \ge 0$.
 \label{stabilitythm}
 \end{thm}
 
 \begin{proof} Notice that if $L > 2\pi/\sqrt{B_4}$, then $B=(2\pi/L)^2$ satisfies $0 < B < B_4$.   Also, as noted above after equation \eqref{Strich},  the existence of a maximizing function for $J_{B,1}$ is equivalent to the existence of a maximizing function for $J_{B,\lambda}$  for every $\lambda > 0$.   Therefore it follows immediately from  Lemma \ref{dilation} and Corollary \ref{existencerange} that $S_{L,\lambda}$ is nonempty.  Furthermore, from Theorem \ref{mainthm} and Lemma \ref{dilation} it also follows that for every minimizing sequence for $S_{L,\lambda}$, one can find a subsequence which, after translations in Fourier space, converges in $L^2_{\rm per}(0,L)$ to a function in $S_{L,\lambda}$.
 
 The stability of the set $S_{L,\lambda}$ follows from a standard argument,  which we summarize here (more details, for example, can be found in \cite{adekoya}).  Suppose, to the contrary, that the set $S_{L,\lambda}$ is not stable.   Then one must be able to find some $\epsilon_0 > 0$, some sequence of initial data $\{u_{0n}\}$ in $L^2_{\rm per}(0,L)$ with corresponding solutions $\{u_n(x,t)\}$, and some sequence of times $\{t_n\}$ in $(0,\infty)$ such that $d(u_{0n},S_{L,\lambda}) \to 0$ as $n \to \infty$ and $d(u_n(\cdot,t_n),S_{L,\lambda}) \ge \epsilon_0$ for all $n \in \mathbb N$.  The assumption on the initial data $\{u_{0n}\}$ implies that by choosing a sequence $\{\alpha_n\}$ in $(0,\infty)$ with $\lim_{n \to \infty} \alpha_n = 1$ such that $P(\alpha_n u_{0n})=\lambda$ for all sufficiently large $n$, we can obtain a minimizing sequence $\{\alpha_n u_{0n}\}$  for $S_{L,\lambda}$.  Moreover, since $H_L$ and $P$ are conserved functionals for \eqref{DMNLSperL}, $\{\alpha_n u_n(\cdot,t_n)\}$ is also a minimizing sequence for $S_{L,\lambda}$.   Therefore there exists a subsequence of  $\{\alpha_n u_n(\cdot,t_n)\}$  which, after translations in Fourier space, converges in $L^2_{\rm per}(0,L)$ to a function in $S_{L,\lambda}$. Since $S_{L,\lambda}$ is invariant under the action of translation in Fourier space, it follows that $d(\alpha_n u_n(\cdot, t_n),S_{L,\lambda})$, and hence also $d(u_n(\cdot,t_n),S_{L,\lambda})$, converges to zero as $n \to \infty$. But this contradicts the assertion that $d(u_n(\cdot,t_n),S_{L,\lambda}) \ge \epsilon_0$ for all $n \in \mathbb N$. 

 \end{proof}
 
 We remark that similar results on the stability of sets of ground-state solutions of the nonlinear Schr\"odinger equation $iu_t + u_{xx} + |u|^p u_x = 0$ date back to the work of Cazenave and Lions in \cite{cazenavelions}. In fact, for the nonlinear Schr\"odinger equation, Cazenave and Lions prove a stronger form of stability called {\it orbital stability}:  namely, they show for a given ground-state profile, the two-dimensional set $\{e^{i\theta}\phi(x+x_0): \theta \in \mathbb R, x_0 \in \mathbb R\}$ is stable in the above sense.  (Note that the term ``orbital stability'' is slightly inaccurate here, in that the orbit in the usual sense of the ground-state solution would be the one-dimensional set $\{e^{i\theta}\phi(x): \theta \in \mathbb R\}$. It is easy to see, however, that this one-dimensional set is not stable in the above sense, cf.\ Remark 8.3.3 on p.\ 274 of \cite{cazenave}.)  In order to prove this stronger form of stability, one generally needs more information on the structure of the set of minimizers of the variational problem.  In the case of the nonlinear Schr\"odinger equation, it follows from the elementary theory of ordinary differential equations that the ground-state profile for a given $L^2$ norm is unique up to translations and multiplications by phase shifts $e^{i\theta}$, which allows one to deduce orbital stability.  However, no such uniqueness result is available yet for the DMNLS equation.
 
In light of the fact that ground-state solutions for the nonlinear Schrodinger equation have, up to symmetries, profiles that are real-valued even functions of $x$, it is interesting to note that at least for some values of $B$,  ground-state solutions of the periodic DMNLS equation cannot have real-valued even profiles:

\begin{corollary}
In the case $L=\sqrt{8\pi}$, the set $S_{L,\lambda}$ of ground-state profiles is nonempty for every $\lambda > 0$.  However, none of the the functions in $S_{L,\lambda}$ are real-valued and even.
\end{corollary}

\begin{proof}
The assertion that $S_{\sqrt{8\pi},\lambda}$ is nonempty follows from  Corollary \ref{existencerange} and Lemma \ref{dilation}.

Suppose now that $\psi \in S_{\sqrt{8\pi},\lambda}$.  Then from Lemma \ref{dilation} we have that $v=M_{\delta} \psi \in L^2(\mathbb T)$ is a maximizer for $J_{B,\lambda/\delta}$, where $B=\pi/2$ and $\delta =\sqrt{2/\pi}$.   From \eqref{DBrealsimp} we see that in the case $B=\pi/2$,  for every function $u \in L^2(\mathbb T)$ such that $\hat u(n)$ is real-valued for all $n \in \mathbb N$, we have $A_B(u)=-a_{0,0}(u)<0$.    On the other hand, from Corollary \ref{ABcriterion} and its proof one sees that for a maximizer $v$ for $J_{B,\lambda/\delta}$, one necessarily has $A_B(v) > 0$.  Therefore the Fourier coefficients of $v$ cannot be real-valued.  Since real-valued even functions must have real-valued Fourier coefficients, it follows that $v$ cannot be real-valued and even.  Therefore $\psi$ cannot be real-valued and even either.
 \end{proof}

We conclude with an easy nonexistence result, which shows in particular that Theorem \ref{stabilitythm} cannot be extended to all positive values of $L$.

\begin{thm} 
If $L= 2\sqrt{\pi/N}$ for some $N \in \mathbb N$, then $S_{L,\lambda}$ is empty for every $\lambda > 0$.  Hence, for these values of $L$, the periodic DMNLS equation \eqref{DMNLSperL} has no ground-state solutions.
\label{nogroundstate}
\end{thm}

\begin{proof} Suppose $L = 2\sqrt{\pi/N}$ for some $N \in \mathbb N$, and $\lambda > 0$.  From Lemma \ref{dilation}, we see that  a function $\psi \in L^2_{\rm per}(0,L)$ can be in $S_{L,\lambda}$ only if $M_\delta(\psi) \in L^2(\mathbb T)$ is a maximizer for $J_{2\pi,\lambda/\delta}$, where $\delta = 1/\sqrt{N\pi}$. But from Corollary \ref{nonexistence} we know that $J_{N\pi,1}$, and hence also $J_{N\pi,\lambda/\delta}$, can have no maximizers. Therefore $S_{L,\lambda}$ must be empty. 
\end{proof}

\end{document}